\documentclass{birkjour}
\usepackage{geometry}
\usepackage[english]{babel}
\usepackage{graphicx}
\usepackage{amsmath}
\usepackage{amsfonts}
\usepackage{xcolor}

\begin{document}

\bibliographystyle{spmpsci}

\newcommand{\wk}{\mbox{$\,<$\hspace{-5pt}\footnotesize )$\,$}}
\newcommand{\smallslash}{\smaller[4]\mbox{$^{_{\scalebox{1.18}{/}}}$}}
\newcommand{\ndashv}{\mbox{$\smallslash$\hspace{-6.4pt}$\dashv$}}

\numberwithin{equation}{section}
\newtheorem{teo}{Theorem}
\newtheorem{lemma}{Lemma}

\newtheorem{coro}{Corollary}
\newtheorem{prop}{Proposition}
\theoremstyle{remark}
\newtheorem{remark}{Remark}
\newtheorem{scho}{Scholium}
\theoremstyle{definition}
\newtheorem{defi}{Definition}
\numberwithin{lemma}{subsection}
\numberwithin{prop}{subsection}
\numberwithin{teo}{subsection}
\numberwithin{defi}{subsection}
\numberwithin{coro}{subsection}
\numberwithin{figure}{subsection}

\numberwithin{remark}{subsection}
\numberwithin{scho}{subsection}
\newtheorem{open}{Open Problem}

\title{Angles in normed spaces}

\author[V. Balestro]{Vitor Balestro}
\address{CEFET/RJ Campus Nova Friburgo
\newline
28635-000 Nova Friburgo
\newline
Brazil
\newline
\&
\newline
Instituto de Matem\'{a}tica e Estat\'{i}stica
\newline
Universidade Federal Fluminense
\newline
24020-140 Niter\'{o}i
 \newline
Brazil}
\email{vitorbalestro@mat.uff.br}
\author[\'{A}. G. Horv\'{a}th]{\'{A}kos G. Horv\'{a}th}
\address{Institute of Mathematics
\newline
Budapest University of Technology and Economics
\newline
1111 Budapest
\newline
Hungary}
\email{ghorvath@math.bme.hu}
\author[H. Martini]{Horst Martini}
\address{Fakult\"at f\"ur Mathematik
\newline
Technische Universit\"at Chemnitz
\newline
09107 Chemnitz
\newline
Germany
\newline
\&
\newline
Department of Applied Mathematics
 \newline
Harbin University of Science and Technology
\newline
150080 Harbin
\newline
China}
\email{martini@mathematik.tu-chemnitz.de}
\author[R. Teixeira]{Ralph Teixeira}
\address{Instituto de Matem\'{a}tica e Estat\'{i}stica
\newline
Universidade Federal Fluminense
\newline
24020-140 Niter\'{o}i
 \newline
Brazil}
\email{ralph@mat.uff.br}

\begin{abstract} The concept of angle, angle functions, and the question how to measure angles present old and well-established mathematical topics referring to Euclidean space, and there exist also various extensions to non-Euclidean spaces of different types. In particular, it is very interesting to investigate or to combine (geometric) properties of possible concepts of angle functions and angle measures in finite-dimensional real Banach spaces (= Minkowski spaces). However, going into this direction one will observe that there is no monograph or survey reflecting the complete picture of the existing literature on such concepts in a satisfying manner. We try to close this gap. In this expository paper (containing also new results, and new proofs of known results) the reader will get a comprehensive overview of this field, including also further related aspects. For example, angular bisectors, their applications, and angle types which preserve certain kinds of orthogonality are discussed. The latter aspect yields, of course, an interesting link to the large variety of orthogonality types in such spaces.
\end{abstract}

\thanks{The first named author thanks to CAPES for partial financial support during the preparation of this manuscript}

\maketitle
\tableofcontents

\section{Introduction} \label{secintro}

\subsection{History}

The concept of angle and the question how to measure angles are old and interesting mathematical topics. It is clear that without using it, the pyramids and other monumental  buildings of the ancient world could not have  been built. Perhaps the first serious contribution, which examined properties of angles, arised from the Greek philosopher Eudemus of Rhodes (c. 370 BC -- c. 300 BC) whose paper \emph{On the angle} clarified   Aristotle's work on this concept (Aristotle was also his teacher). Unfortunately, this work is lost, but according to Proclus (412--485) (who read it) the concept of the angle by Eudemus is a quality that he regarded as deviation from a straight line (see in \cite{wehrli}). Proclus also quotes the view of Carpus of Antioch (between 2nd century BC and 2nd century AD) who regarded it as the interval or domain between the intersecting lines. Hence, according to him the angle is a quantity. For the interested reader we propose to read \emph{Proclus' Commentary on Euclid I} (see \cite{proclus}). From this source we also know that Euclid (of Alexandria, ~ c. 300 BC) adopted a third concept in which the angle is a relationship, besides his definitions of right, acute and obtuse angles. He defined perpendicularity of two lines and right angles via the four equal domains determined by them in its $10$-th definition, and he says then the obtuse angle is that which is larger than the right angle, and the acute angle the respectively smaller one (see \cite{euclid}).

During the long reign period of Euclidean geometry it was not important to ask how we can measure an angle. There were only few important situations when it was required to compare  corresponding pairs of lines
with concrete values, for example when the two half-lines (which are the legs of the angle) are concurrent, are complements with respect to a line, or divide the plane into general comparable parts. In these simple cases it is clear that if we define the measure the angle of two complementary half-lines of a line, then we get from it a measure for all angles via a rational operation. In practice, two methods spread. The first one arised (probably) from the Babylonians who knew that the perimeter of a hexagon was exactly equal to six times the radius of a circumscribed circle of it, a fact that was evidently the reason why they decided to divide the circle into $360$ parts (from the practical methods to the history of measuring angles in this way we propose to read the interesting paper of \cite{wallis}). Observe that if we measure an angle in degrees, at the same time one can realise an area-based and an arc-length based measurement of it. Despite this, for a long time the principle was known that the measure of an angle can be defined on the base of the ratio of the corresponding length of the circular arc and the radius of the circle, the concept of radian measure. This is opposed to the degree of an angle, normally credited to Cotes in the early 17th century (see \cite{cotes}.) Because radians are more mathematical, this led to a more elegant formulation of important results. Thus, in recent language of mathematics angles are universally measured in radians. The most inspirative example is the so-called Euler identity which says that if the inputs of the trigonometric functions $\sin$ and $\cos$ are given in radians, then $e^{ix}=\cos x +i\sin x$ holds for any real number $x$.

In the eighteenth and in the nineteenth century the thought that Euclidean geometry is the only possibility gradually lost its validity. If became obvious that in various (essentially Euclidean) situations  the concept of angle and the measurement of the angles should be reconsidered. Spherical geometry was the first non-Euclidean geometry in which the inner concept of the angle seriously required that we define and investigate the angles of space curves which are no longer Euclidean straight lines. This problem was solved easily, because there is a natural embedding of the sphere into Euclidean 3-space and the concept of angle can be transformed into the concept of planes of the space. This way was not possible in hyperbolic geometry, because there is no respective embedding into Euclidean $3$-space. Hence the problem to define angles required an axiomatic foundation of the used concept of angles, giving then the base of a more general definition. Hilbert's axiomatic approach to geometry (see \cite{Hil}) gave this frame. Many authors dealt with this problem in a large variety of interesting situations. Not giving a complete list, we mention here some important methods.

For continuously differentiable curves satisfying a general extremal property, the concept of angle was discussed by Bliss \cite{Bliss}. He defined his concept as follows: If $OA'$ and $OA$ are two extremal rays through the point $O$, and
$A'A$ is the arc of length $l$ of a transversal (which is a generalized circle passing through $A$ and $A'$, with center $O$) at the generalized distance $r$ from
$O$, then the generalized angle between $OA'$ and $OA$ is defined to be the limit
of the ratio $l/r$ as $r$ approaches zero. His analytical formulas reflect the usual computation methods in classical Euclidean and non-Euclidean geometries, and also for surfaces embedded into the Euclidean $3$-space.

When Minkowski introduced the concept of the norm function, it was immediately clear that in a normed plane both the concepts of orthogonality of lines and of angle measure should be revised. Since there was no reason to change the concept of angle domain, this situation was the first in the literature when the measurement of the angle and the definition of the angle domain were automatically separated. The difficulties here were communicated by the great differential geometer H. Busemann. In \cite{Bus2} he investigated the geometry of the so-called Finsler spaces, and he observed the following facts:
\emph{The volume problem makes it more than probable that an analogous
situation\footnote{as the Riemannian one} exists for Finsler spaces. Therefore the study of Minkowskian
geometry ought to be the first and main step, the passage from there to
general Finsler spaces will be the second and simpler step.
What has been done in Minkowskian geometry, what are the difficulties
and problems, and which tools will be necessary? Little has
been done, but the field is quite accessible. The main difficulty comes
from our long Euclidean tradition, which makes it hard (at least for
the author) to get a feeling for the subject and to conjecture the
right theorems.
The type of problem which faces us is clear: A Minkowskian geometry
admits in general only the translations as motions and not the
rotations. Since the group of motions is smaller, we expect more invariants.
By passing from Euclidean to projective geometry, ellipses,
parabolas, and hyperbolas become indiscernible. The present case
presents \texttt{ the much more difficult converse problem, to discern objects
which have always been considered as identical}.}

In \cite{Bus1}, Busemann discussed the "axiom" for angle measures in the case of plane curves belonging to a class $\mathcal{S}$ of open Jordan curves, holding the additional property that any two distinct points lie on exactly one curve of $\mathcal{S}$. He defined the notions of \textit{ray} $r$, \textit{angle} $D$ \emph{with} \emph{legs} $r_1$ \emph{and} $r_2$, and \textit{angle measure} $|D|$ on the set of angles having the following properties:
\begin{enumerate}
\item $|D|\geq 0$ (positivity),
\item $|D|=\pi$ if and only if $D$ is straight,
\item if $D_1$ and $D_2$ are two angles with a common leg but with no other common ray, then $|D_1\cup D_2|=|D_1|+|D_2|$ (additivity),
\item if $D_\nu \rightarrow D$, then $|D_\nu|\rightarrow D$ (continuity).
\end{enumerate}
He showed that these assumptions are sufficient to obtain many of the usual relationships
between angle measure and curvature. We note that Busemann collected the essential properties of an angle measure  that we have to require in every structure, where a natural concept of angle exists.

Lippmann \cite{Li1} considered the classical Minkowski space which was defined on the $n$-dimensional Euclidean space by a "metrische Grundfunction" $F$ which is a positive, convex functional on the space that is
homogeneous of first degree. In our terminology, $F$ is the norm-square function. To have convexity (following Minkowski's definition), Lippmann required continuity of the second partial derivative, and positivity of the second derivative of $F$. Hence the unit ball of the corresponding space is always smooth. He used the arcus cosine of the bivariate function
$$
(x,y):=\frac{\sum x_i\frac{\partial }{\partial x_i}F(y)}{F(x)}
$$
to measure the angle between $x$ and $y$. This yields a concept of transversality, namely: $x$ is \emph{transversal} to $y$ if $(x, y)=0$. A wide variety
of angle measures referring to metric properties can be found in the literature. E.g., Lippmann's papers \cite{Li2,Li3} contain typically metric definitions of angle measures. For the situation in (normed or) Minkowski planes we refer, in addition, to the papers already mentioned and Graham, Witsenhausen and Zassenhaus \cite{G-W-Z}. This paper refers to a useful metric classification of angles by their measures, and a good review on this topic can be found in the book of Thompson \cite{Tho}.

In the last few decades some authors rediscovered this interesting problem in connection with the problem of orthogonality. We have to mention P. Brass who in \cite{brass} redefined the concept of angle measure: an \textit{angle measure} is a measure $\mu$ on the unit circle $\partial B$ with center $O$ which is extended in the usual translation-invariant way to measure angles elsewhere, and which has the following properties:
\begin{enumerate}
\item $\mu(\partial B) = 2\pi$,
\item for any Borel set $S\subset \partial B$ we have $\mu(S) = \mu(-S)$, and
\item for each $p \in \partial B$ we have $\mu(\{p\}) = 0$.
\end{enumerate}

This concept was used in the papers of D\"uvelmeyer \cite{Duev}, Martini and Swanepoel \cite{martiniantinorms}, and Fankh\"anel \cite{fankhanel2009i,Fank}.

Another direction of research is to give immediate metric definitions of the angle of two vectors. In this direction we can find  papers of P. M. Mili\v{c}i\v{c} \cite{milicic2007b}, C. R. Diminnie, E. Z. Andalafte, R. W. Freese \cite{diminnie1988generalized,D-A-F} or H. Gunawan, J. Lindiarni, and O. Neswan \cite{gunawan2008p}. Further related papers on angle measures are \cite{Dek1}, \cite{Dek2}, \cite{Dek3}, \cite{Ling} \cite{Gol1} and \cite{Gol2}. It is our aim to survey such concepts. Definitions and main results are presented, sometimes with more detailed (and sometimes different) proofs. A new angle function  (namely, the S-angle) is also presented, and the definition of the B-angle is extended to a larger class of Minkowski spaces. Moreover, to our best knowledge the results in Lemma \ref{degeneracypangle}, Proposition \ref{homogeneitypangle}, Lemma \ref{degeneracyiangle}, Proposition \ref{homogeneityiangle}, Theorem \ref{thureytheorem} as well as a proof of continuity of the B-angle are new.

\subsection{Notation and basic concepts}
Throughout the text, $(X,||\cdot||)$ will denote a real normed space. We denote the \emph{closed segment} connecting two points $x,y \in X$ by $\mathrm{seg}[x,y]$, and the \emph{open segment} joining the same points by $\mathrm{seg}(x,y)$. For the \emph{line} spanned by $x,y \in X$ we write $\left<x,y\right>$. The \emph{unit ball} and the \emph{unit sphere} of $(X,\|\cdot\|)$ are denoted by $B$ and $S$, respectively. Furthermore, $B(\rho)$ and $S(\rho)$ stand for the \emph{ball} and for the \emph{sphere of radius} $\rho$ \emph{centered at the origin}. We introduce now the orthogonality concepts that will be used in the text. Two vectors $x,y \in X$ are said to be \\

\noindent$\bullet$ \emph{Birkhoff orthogonal} (denoted $x \dashv_B y$) when $||x+ty|| \geq ||x||$ for any $t \in \mathbb{R}$,\\
$\bullet$ \emph{isosceles orthogonal} (denoted $x \dashv_I y$) if $||x+y|| = ||x-y||$, \\
$\bullet$ \emph{Phytagorean orthogonal} (denoted $x \dashv_P y$) whenever $||x||^2 + ||y||^2 = ||x-y||^2$, and\\
$\bullet$ \emph{Singer orthogonal} (denoted $x \dashv_S y$) if either $||x||\cdot||y|| = 0$ or $\left|\left|\frac{x}{||x||}-\frac{y}{||y||}\right|\right| = \left|\left|\frac{x}{||x||}+\frac{y}{||y||}\right|\right|$.\\

    A comprehensive paper on orthogonality types in normed spaces is \cite{alonso}. Some other orthogonality types (namely, the ones obtained by the functional $g$) will be also dealt with in this paper; these will be defined later. \\

A Minkowski plane is said to be a \emph{Radon plane} if Birkhoff orthogonality is symmetric. In this case we say that its unit circle is a \emph{Radon curve}. Fixing a determinant form $[\cdot,\cdot]$ in the plane (which is unique up to constant multiplication), we get a new norm by setting
\begin{align*} ||x||_a = \sup\{|[x,y]|:y\in S\}. \end{align*}
We call it the \emph{antinorm} of $\|\cdot\|$. A plane is Radon if and only if the norm and the antinorm are proportional (see \cite{martiniantinorms}). In this case we always can consider the determinant form rescaled in such a way that we have $||\cdot|| = ||\cdot||_a$.\\

The book \cite{Tho} as well as the surveys \cite{martini2} and \cite{martini1} are basic references for Minkowski geometry at large.

\section{Angle functions}\label{abstractangles}

\subsection{Angle axioms}

One of the ways to define an angle concept in a normed space $(X,||\cdot||)$ is to consider a function $\mathrm{ang}:X_{o}\times X_{o} \rightarrow [0,\pi]$, where $X_{o}$ stands for the set $X\setminus \{o\}$. We will refer to this type of functions as \textit{angle functions}. Given such a function, we may define the \textit{measure of an angle} $\wk\mathbf{xyz}$ formed by three points to be $\mathrm{ang}(x-y,z-y)$. At first glance, we will not demand angle functions to have any properties; the method adopted in this expository paper is to enunciate some ``good" properties expected for angle functions and to study all the angle concepts regarding these properties. Some of these general properties already appeared in \cite{thurey1}, yielding the notion of \emph{angle spaces} for normed spaces endowed with some ``well behaving" angle function. First, we enunciate the properties that we believe to be the most ``natural" ones for angle functions: \\

\noindent\textbf{1.} For each fixed $x \in X_{o}$, the functions $y \mapsto\mathrm{ang}(x,y)$ and $y \mapsto \mathrm{ang}(y,x)$ are continuous surjective functions of $X_o$ \textit{(continuity)}; \\
\noindent\textbf{2.} $\mathrm{ang}(x,y) = \mathrm{ang}(y,x)$ for every $x,y \in X_{o}$ \textit{(symmetry)};\\
\noindent\textbf{3.} $\mathrm{ang}(\alpha x,\beta y) = \mathrm{ang}(x,y)$ for all $x,y \in X_{o}$ and $\alpha,\beta >0$ \textit{(homogeneity)};\\
\noindent\textbf{4.} $\mathrm{ang}(x,\alpha x+ \beta y)+\mathrm{ang}(\alpha x + \beta y,y) = \mathrm{ang}(x,y)$ for every $x,y \in X_{o}$ and $\alpha,\beta >0$ \textit{(additivity)}; and \\
\noindent\textbf{5.} $\mathrm{ang}(x,y) = 0$ if and only if $x = \alpha y$ for some $\alpha > 0$ \textit{(non-degeneracy)}.\\

We will refer to these five axioms as the \textit{structural axioms}, and we define now some other (not unexpected) properties that an angle function might have. We will call them \textit{positional properties}.\\

\noindent\textbf{6.} $\mathrm{ang}(x,y) = \pi$ if and only if $y = \alpha x$ for some $\alpha < 0$ \textit{(parallelism);}  \\
\noindent\textbf{7.} $\mathrm{ang}(x,y) + \mathrm{ang}(y,-x) = \pi$ \textit{(supplementarity)}; and \\
\noindent\textbf{8.} $\mathrm{ang}(x,y) = \mathrm{ang}(-x,-y)$ \textit{(opposite invariance)}.

\begin{lemma}\label{properties} Let $\mathrm{ang}:X_{o}\times X_{o} \rightarrow [0,\pi]$ be an angle function in a normed space $(X,||\cdot||)$. If $\mathrm{ang}$ satisfies the structural axioms, then the positional properties hold.
\end{lemma}
\begin{proof} For property 6, let $x,y \in X_{o}$ be such that $\mathrm{ang}(x,y) = \pi$ and assume that $x$ and $y$ are independent. Set $z = 2y - x$. Hence $y = \frac{1}{2}x+\frac{1}{2}z$, and thus, from the structual axioms 4 and 5, we have $\mathrm{ang}(x,y) <\mathrm{ang}(x,z) \leq \pi$. This is a contradiction, and therefore $x$ and $y$ cannot be independent. The structural axiom 5 yields that we must have $x = \alpha y$ for some $\alpha < 0$. The converse is trivial. \\

For property 7, assume that $x,y \in X_{o}$ are independent (the other case is immediate). Define a sequence of vectors $(z_n)_{n\in\mathbb{N}}$ given by $z_n = \frac{1}{n}y + \left(\frac{1}{n}-1\right)x$. Since $y = nz_n + (n-1)x$, it follows from the additivity axiom that $\mathrm{ang}(x,y) + \mathrm{ang}(y,z_n) = \mathrm{ang}(x,z_n)$. With $n\rightarrow\infty$ the desired follows from the continuity axiom.\\

Property 8 is an easy consequence of property 7.

\end{proof}

We also consider the following two properties based on criteria for congruent triangles in Euclidean spaces:  \\

\noindent\textbf{9.} If $x,y,v,w \in X_{o}$ are such that $||x|| = ||v||$, $||y|| = ||w||$, $\mathrm{ang}(x,y) = \mathrm{ang}(v,w)$, and $x \neq y$, then $v \neq w$, and $\mathrm{ang}(x-y,-y) = \mathrm{ang}(v-w,-w)$. \\
\noindent\textbf{10.} If $x,y,v,w \in X_{o}$ are such that $||x|| = ||v||$, $||y|| = ||w||$, $\mathrm{ang}(x,y) = \mathrm{ang}(v,w)$, and $x \neq y$, then $v \neq w$ and $||x-y|| = ||v-w||$.\\

We will call them \textit{congruence properties}. Notice that they are not so natural as the structural axioms, and for that reason we prefer once more to use the word ``properties" instead of ``axioms". The next lemma states that they are practically equivalent.

\begin{lemma}\label{congequiv} If $\mathrm{ang}:X_{o}\times X_{o} \rightarrow [0,\pi]$ is an angle function which satisfies all the structural axioms, then properties \normalfont 9 \textit{and} 10 \textit{are equivalent.}
\end{lemma}
\begin{proof} Assume first that property 9 holds. Let $x,y,v,w \in X_{o}$ be as in the hypothesis of property 10, but suppose (without loss of generality) that $||x-y|| > ||v-w||$. Thus, we may find a point $p \in \mathrm{seg}(x,y)$ such that $||y-p|| = ||v-w||$. Clearly, we must have $\mathrm{ang}(p,y) \neq \mathrm{ang}(x,y)$, and hence the triangles $\Delta\mathbf{poy}$ and $\Delta\mathbf{vow}$
yield a contradiction to property 9.\\

Assume now that property 10 holds. Let $x,y,v,w \in X_{o}$ be as in the hypothesis of property 9, but assume that $\mathrm{ang}(x-y,-y) > \mathrm{ang}(v-w,-w)$.  Thus, there is a point $p \in \mathrm{seg}(o,x)$ for which $\mathrm{ang}(p-y,-y) = \mathrm{ang}(v-w,-w)$. Since $||p|| < ||x||=||v||$, property 10 does not hold for the triangles $\Delta\mathbf{poy}$ and $\Delta\mathbf{vow}$.

\end{proof}

\subsection{Characterizing inner product spaces}

In \cite{diminnie1988generalized} it is proved that any normed space $(X,||\cdot||)$ for which there exists an angle function $\mathrm{ang}:X_{o}\times X_{o} \rightarrow [0,\pi]$ satisfying the structural axioms and also the property 9 (or, equivalently, property 10) must be an inner product space. Moreover, this angle function must be the standard Euclidean one. We are going to prove this fact.

\begin{teo}\label{angleinnerproduct} Let $(X,||\cdot||,\mathrm{ang})$ be a normed space endowed with an angle function which satisfies all the structural axioms and the congruence properties. Then the norm $||\cdot||$ is derived from an inner product.
\end{teo}

\begin{proof} We prove that under the given assumptions isosceles orthogonality implies Birkhoff orthogonality. This is a known characterization of inner product spaces (see \cite{alonso}). Indeed, let $x,y \in X$ be nonzero vectors that are isosceles orthogonal, and define the function $f:[-1,1]\rightarrow\mathbb{R}$ by $f(t) = \mathrm{ang}(x-y,tx-y) - \mathrm{ang}(tx-y,-x-y)$. It is easy to see that $f(-1) > 0$ and $f(1) < 0$, and hence, by continuity, there exists a number $t_0 \in (-1,1)$ with $f(t_0) = 0$. This means that $\mathrm{ang}(x-y,t_0x-y) = \mathrm{ang}(t_0x-y,-x-y)$, and property 10 gives $||(t_0x-y)-(x-y)|| = ||(t_0x-y)-(-x-y)||$. It follows immediately that $t_0 = 0$, and then $\mathrm{ang}(x-y,-y) = \mathrm{ang}(-y,-x-y)$. From property 9 we have $\mathrm{ang}(x,y) = \mathrm{ang}(x,-y)$, and hence homogeneity gives $\mathrm{ang}(\lambda x,y) = \mathrm{ang}(\lambda x, -y)$ for every $\lambda \in \mathbb{R}$. Applying again property 10, we therefore have that $x$ and $y$ are Roberts orthogonal, and thus also Birkhoff orthogonal.
The proof is finished.
\end{proof}

We prove now that any angle function satisfying all the structural axioms and congruence properties is the standard Euclidean angle. For this purpose we first need the following lemma.

\begin{lemma} \label{lemmaequallength} Let $\mathrm{ang}:X_{o}\times X_{o} \rightarrow \mathbb{R}$ be an angle function defined in a normed space $(X,||\cdot||)$ and assume that $\mathrm{ang}$ satisfies all the structural axioms and the congruence properties. If $x,y,v,w \in X_{o}$ are such that $||x|| = ||v||$, $||y|| = ||w||$, and $||x-y|| = ||v-w||$, then $\mathrm{ang}(x,y) = \mathrm{ang}(v,w)$.
\end{lemma}

\begin{proof} Assume that the congruence properties hold, and let $x,y,v,w \in X$ be non-zero vectors with $||x|| = ||v||$, $||y|| = ||w||$ and $||x-y|| = ||v-w||$, but suppose that $\mathrm{ang}(x,y) > \mathrm{ang}(v,w)$. We may choose $p \in \wk\mathbf{xoy}$ such that $||p|| = ||y||$ and $\mathrm{ang}(x,p) = \mathrm{ang}(v,w)$. From the congruence property 10 we have $||x-p|| = ||v-w||$, and thus $p \notin \mathrm{seg}[x,y]$. Indeed, $y$ is the only point of $\mathrm{seg}[x,y]$ such that $||x-y|| = ||v-w||$, but clearly we have $p \neq y$. We have two possibilities: first, if the points $o,x,y$ and $p$ are vertices of a convex quadrilateral, then this quadrilateral is such that the sum of the lengths of the diagonals equals the sum of the lengths of two opposite sides, and this contradicts the strict convexity hypothesis (see \cite{martini1}). The other possible configuration is that $p \in \mathrm{int}(\mathrm{conv}(o,x,y))$, and in this case we also obtain a contradiction to the hypothesis of strict convexity. \\
\end{proof}

Now we come to the announced theorem.

\begin{teo}\label{euclideanspace} Let $(X,(\cdot,\cdot))$ be an inner product space. An angle function $\mathrm{ang}:X_{o}\times X_{o}\rightarrow\mathbb{R}$ which satisfies the angle axioms and the congruence properties is necessarily the usual Euclidean angle function. In other words, we have then
\begin{align*}
\mathrm{ang}(x,y) = \mathrm{arccos}\frac{(x,y)}{||x|| \cdot ||y||}
\end{align*}
for every $x,y \in X_{o}$, where $\mathrm{arccos}:[-1,1]\rightarrow [0,\pi]$ is the inverse of the restriction of the standard cosine function to the interval $[0,\pi]$.
\end{teo}
\begin{proof} For simplicity, we introduce $\mathrm{ang_e}(x,y) = \mathrm{arccos}\frac{(x,y)}{||x||\cdot||y||}$. Since both $\mathrm{ang}$ and $\mathrm{ang_e}$ are additive (in the sense of the angle axiom 4) and continuous, it  suffices to prove that $\mathrm{ang}(x,y) = \frac{\pi}{2^n}$ if and only if $\mathrm{ang_e}(x,y) = \frac{\pi}{2^n}$ for every $n \in \mathbb{N}$. We prove this via induction. Let $x,y \in X_{o}$ be such that $\mathrm{ang}(x,y) = \frac{\pi}{2}$. In this case, we clearly have $\mathrm{ang}(-x,y) = \frac{\pi}{2}$, and from the congruence property 10 it follows that $||x-y|| = ||x+y||$. Hence, the triangles $\Delta\mathbf{oxy}$ and $\Delta\mathbf{o(-x)y}$ have sides with correspondingly equal lengths, and it follows that $\mathrm{ang_e}(x,y) = \mathrm{ang_e}(-x,y)$, and thus $\mathrm{ang_e}(x,y) = \frac{\pi}{2}$. Conversely, if $\mathrm{ang_e}(x,y) = \frac{\pi}{2}$, we have immediately $||x+y|| = ||x-y||$, and from Lemma \ref{lemmaequallength} we have $\mathrm{ang}(x,y) = \mathrm{ang}(-x,y)$. This yields $\mathrm{ang}(x,y) = \frac{\pi}{2}$. To prove the validity of the hypothesis for $n+1$ when it holds for $n$, we use exactly the same strategy. This finishes the proof.

\end{proof}

\section{Angles preserving orthogonality types}

We say that an angle function defined in a normed space preserves a certain orthogonality type when it attains the ``Euclidean value" $\frac{\pi}{2}$ only for orthogonal pairs of vectors. Gunawan, Lindiarni and Neswan (cf. \cite{gunawan2008p}) defined two angle functions preserving orthogonality types: the \textit{P-angle}, which preserves Pythagorean orthogonality, and the \textit{I-angle}, which preserves isosceles orthogonality. Th\"{u}rey (cf. \cite{thurey1}) defined the \textit{Thy-angle} preserving Singer orthogonality. For angles preserving Birkhoff orthogonality we have the notion of \textit{q-angle}, defined by Zhi-Zhi, Wei and L\"{u}-Lin (cf. \cite{zhi2011projections}). We suggest and study a new angle type which is a modification of the q-angle, and call it \textit{S-angle}. Mili\v{c}i\v{c} (see below) introduced the \textit{B-angle}, which also preserves Birkhoff orthogonality, and an angle based on the Gateaux derivative of the norm. This section is devoted to investigate these angle functions.

\subsection{The P-angle} \label{pangle}

 Inspired by Phythagorean orthogonality, we define the \textit{P-angle} between two non-zero vectors $x,y \in X$ to be
\begin{align*} \mathrm{ang_p}(x,y) = \mathrm{arccos}\left(\frac{||x||^2+||y||^2-||x-y||^2}{2||x||\cdot||y||}\right) \end{align*}
Notice that this definition is obviously related to the well known cosine law. Namely, we are defining the angle between $x$ and $y$ to be such that $||x-y||^2 = ||x||^2 + ||y||^2 - 2||x||\cdot||y||\cos(\mathrm{ang_p}(x,y))$. Moreover, it is clear that the correctness of this definition is provided by the triangle inequality and the fact that $\mathrm{ang_p}$ is the usual Euclidean angle if the norm is derived from an inner product. \\

In order to be coherent with our treatment of angle functions, we will formally define the P-angle to be the function $\mathrm{ang_p}:X_{o}\times X_{o} \rightarrow \mathbb{R}$ given as above. We will study now some properties of $\mathrm{ang_p}$. First, it is obvious that $\mathrm{ang_p}$ satisfies the structural axioms 1 (continuity) and 2 (symmetry), as well as the positional property 8 (opposite invariance).

\begin{lemma}\label{degeneracypangle} Let $(X,||\cdot||)$ be a normed space with associated P-angle $\mathrm{ang_p}:X_{o}\times X_{o}\rightarrow \mathbb{R}$. Then the following properties are equivalent:\\

\noindent\textbf{(a)} The function $\mathrm{ang_p}$ satisfies the structural axiom 5 (non-degeneracy). \\

\noindent\textbf{(b)} The function $\mathrm{ang_p}$ satisfies the positional property 6 (parallelism).\\

\noindent\textbf{(c)} The space $(X,||\cdot||)$ is strictly convex.\\
\end{lemma}

\begin{proof} It is clear that $\mathrm{ang_p}(x,y) = 0$ if and only if
\begin{align*} \frac{||x||^2+||y||^2-||x-y||^2}{2||x||\cdot||y||} = 1,
\end{align*}
but this is equivalent to $\left|||x||-||y||\right| = ||x-y||$. Also, $\mathrm{ang_p}(x,y) = 0$ if and only if $||x-y|| = ||x|| + ||y||$. The desired equivalences follow immediately.

\end{proof}

\begin{remark}\label{remarkpangle} Notice that even if $(X,||\cdot||)$ is not strictly convex, we still have $\mathrm{ang_p}(x,y) = 0$ and $\mathrm{ang_p}(x,y) = \pi$ when $y = \alpha x$ for $\alpha > 0$ and $\alpha < 0$, respectively.
\end{remark}

\begin{lemma}\label{pangleprop} For $x,y \in X_o$ we have $x \dashv_P y$ if and only if $\mathrm{ang_p}(x,y) = \frac{\pi}{2}$.
\end{lemma}
\begin{proof}
The proof is straightforward.

\end{proof}

\begin{remark}\normalfont It is not necessarily true that $\mathrm{ang_p}(x,y) + \mathrm{ang_p}(x,-y) = \pi$. For example, if we choose the unit circle of $(X,||\cdot||)$ as an affine regular hexagon and $x$ and $y$ as two of its consecutive vertices, then it is easy to see that $\mathrm{ang_p}(x,y) + \mathrm{ang_p}(x,-y) = \frac{4\pi}{3}$.
\end{remark}

\begin{prop}\label{homogeneitypangle} The P-angle function $\mathrm{ang_p}$ of a normed space $(X,||\cdot||)$ satisfies the structural axiom 3 (homogeneity) if and only if the norm $||\cdot||$ is derived from an inner product.
\end{prop}
\begin{proof} In \cite{james1945orthogonality} it was proved that a normed space is an inner product space if and only if we have $x \dashv_P \alpha y$ for every $\alpha \in \mathbb{R}$ whenever $x \dashv_P y$. The proposition follows now with the help of Lemma \ref{pangleprop}.

\end{proof}

\subsection{The I-angle} \label{iangle}

Inspired by this isosceles orthogonality type (and also by the polarization identity for inner product spaces) the I-angle function $\mathrm{ang_i}:X_o\times X_o \rightarrow \mathbb{R}$ is defined by
\begin{align*} \mathrm{ang_i}(x,y) = \mathrm{arccos}\left(\frac{||x+y||^2-||x-y||^2}{4||x||\cdot||y||}\right).
\end{align*}

Again, it follows from the triangle inequality that the definition makes sense, and it is clear that $\mathrm{ang_i}$ is the standard angle if the space is Euclidean. The I-angle shares a lot of properties with the P-angle, as we will see now. First, one can readily see that $\mathrm{ang_i}$ also satisfies the structural axioms 1 (continuity), 2 (symmetry), and 8 (opposite invariance). It is also immediate to check that the I-angle preserves isosceles orthogonality in the sense that two vectors $x,y \in X_o$ are isosceles orthogonal if and only if $\mathrm{ang_i}(x,y) = \frac{\pi}{2}$. We prove now the analogue of Lemma \ref{degeneracypangle} for the notion of I-angle.

\begin{lemma}\label{degeneracyiangle} Let $(X,||\cdot||)$ be a normed space, and let $\mathrm{ang_i}:X_o\times X_o \rightarrow \mathbb{R}$ be its I-angle function. Then the following properties are equivalent: \\

\noindent\normalfont\textbf{(a)}  \textit{The function $\mathrm{ang_i}$ satisfies the structural axiom 5 (non-degeneracy).}\\

\noindent\textbf{(b)}  \textit{The function $\mathrm{ang_i}$ satisfies the positional property 6 (parallelism).}\\

\noindent\textbf{(c)} \textit{The space $(X,||\cdot||)$ is strictly convex.}\\
\end{lemma}
\begin{proof} If $\mathrm{ang_i}(x,y) = 0$, then we must have $||x+y||^2 - ||x-y||^2 = 4||x||\cdot||y||$. From the triangle inequality it follows that $(||x||+||y||)^2 - ||x-y||^2 \geq 4||x||\cdot||y||$, which yields $|||x||-||y||| \geq ||x-y||$. Since the opposite inequality holds, we have $|||x||-||y||| = ||x-y||$, and the desired follows. For $\mathrm{ang_i}(x,y) = \pi$ we can use a similar calculation.

\end{proof}

\begin{remark}\label{remarkiangle} As for the P-angle, we have that $\mathrm{ang_i}(x,\alpha x) = 0$ if $\alpha > 0$ and $\mathrm{ang_i}(x, \alpha x) = \pi$ whenever $\alpha < 0$.
\end{remark}

\begin{prop}\label{homogeneityiangle} Let $(X,||\cdot||)$ be a normed space, and let $\mathrm{ang_i}:X_o\times X_o \rightarrow \mathbb{R}$ be the associated I-angle function. Then $\mathrm{ang_i}$ is homogeneous if and only if the norm is Euclidean.
\end{prop}
\begin{proof} It is known that isosceles orthogonality in a normed space is homogeneous if and only if the norm is derived from an inner product (see \cite{james1945orthogonality}). Since the I-angle preserves isosceles orthogonality, we are finished.

\end{proof}

We emphasized the similarities between the P-angle and the I-angle. However, there is also a difference that is worth to be pointed out, namely: while this does not happen to the P-angle, the I-angle of any normed space has the positional property 7 (supplementarity). Indeed,
\begin{align*} \mathrm{ang_i}(x,-y) = \mathrm{arccos}\left(-\frac{||x+y||^2-||x-y||^2}{4||x||\cdot||y||}\right) = \\ = \pi - \mathrm{arccos}\left(\frac{||x+y||^2-||x-y||^2}{4||x||\cdot||y||}\right) = \pi - \mathrm{ang_i}(x,y).
\end{align*}

\subsection{The Thy-angle} \label{thyangle}

In \cite{thurey1} an angle is defined, which the author calls the \textit{Thy-angle}, and which preserves Singer orthogonality. This is the angle function $\mathrm{ang_{thy}}:X_o\times X_o \rightarrow \mathbb{R}$ given by
\begin{align*} \mathrm{ang_{thy}}(x,y) = \mathrm{arccos}\left(\frac{1}{4}\left(\left|\left|\frac{x}{||x||}+\frac{y}{||y||}\right|\right|^2-\left|\left|\frac{x}{||x||}-\frac{y}{||y||}\right|\right|^2\right)\right)\,,
\end{align*}
which was also studied in \cite{milicic2011thy}. We will enunciate and prove now some of its properties. These results can already be found in the mentioned papers.\\

\begin{lemma}\label{propertiesthyangle} For the Thy-angle the structural axioms 1 (continuity), 2 (symmetry), 3 (homogeneity) and 5 (non-degeneracy) hold, as well as any of the positional properties. Also, two non-zero vectors $x,y \in X_o$ are Singer orthogonal if and only if $\mathrm{ang_{thy}}(x,y) = \frac{\pi}{2}$.
\end{lemma}
 The proof of this lemma is simple, and we will omit it. All the axioms and properties are enunciated in \cite{thurey1}, and they also appear in \cite{milicic2011thy}. The fact that Thy-angle preserves Singer orthogonality was also noticed by Mili\v{c}i\v{c} (see \cite{milicic2011thy}).

\begin{remark}  The Thy-angle has a property which is a little bit stronger than homogeneity. Namely, we have that $\mathrm{ang_{thy}}(\alpha x,\beta y) = \mathrm{sgn}(\alpha\beta)\mathrm{ang_{thy}}(x,y)$ for every $\alpha,\beta \neq 0$.
\end{remark}

Notice also that the Thy-angle clearly coincides with the standard angle function in Euclidean space. We now prove a result from \cite{thurey1}, but based on a more general hypothesis (Th\"{u}rey proves it for the Thy-angle). The reader should be aware of the fact that this theorem justifies the existence of a sort of polar coordinates in an arbitrary normed plane. This is formally explained by Th\"{u}rey \cite[Corollary 1]{thurey1}.\\

\begin{teo}\label{thureytheorem} Let $\mathrm{ang}:X_o\times X_o \rightarrow \mathbb{R}$ be an angle function for which the structural axioms 1 (continuity), 3 (homogeneity), and 5 (non-degeneracy) as well as the positional property 6 (parallelism) hold. Then, if $x,y \in X_o$ are linearly independent vectors and the function $f:t \mapsto \mathrm{ang}(y,x+ty)$ is injective,  $f$ is a decreasing homeomorphism from $\mathbb{R}$ to $(0,\pi)$.
\end{teo}
\begin{proof} The continuity of $f$ comes from the continuity of $\mathrm{ang}(\cdot,\cdot)$. Also, from the homogeneity we can conclude $f(t) = \mathrm{ang}\left(y,\frac{x+ty}{||x+ty||}\right)$, and thus non-degeneracy and parallelism yield $\lim_{t \rightarrow \infty}f(t) = 0$ and $\lim_{t \rightarrow -\infty}f(t) = \pi$, respectively. Indeed, it is easy to see that $\frac{x+ty}{||x+ty||}$ converges to $y$ and $-y$ when $t$ goes to $\infty$ and $-\infty$, respectively. The rest follows from standard analysis results.

\end{proof}

\begin{remark} Notice that if we drop the injectiveness hypothesis, then we still have that $f$ is a non-increasing surjective function from $\mathbb{R}$ onto $(0,\pi)$.
\end{remark}

\begin{coro}\label{thyanglefunction} Let $(X,||\cdot||)$ be a strictly convex normed space with associated Thy-angle $\mathrm{ang_{thy}}:X_o\times X_o \rightarrow \mathbb{R}$. Then, if $x,y \in X_o$ are linearly independent, the function $f:\mathbb{R}\rightarrow (0,\pi)$ given by $f(t)=\mathrm{ang_{thy}}(y,x+ty)$ is a decreasing homeomorphism.
\end{coro}
\begin{proof} We just have to prove that $f$ is strictly decreasing. But this is clearly the same as proving that the function $z \mapsto ||y+z|| - ||y-z||$ is injective when $z$ ranges on the unit circle from $y$ to $-y$. The chief ingredient is the Monotonicity Lemma (see \cite{martini1}).  Let $x,z \in S$ be such that $x$ lies between $z$ and $y$. Hence, using the mentioned lemma, we compare the lengths of the chords $\mathrm{seg}[y,x]$ and $\mathrm{seg}[y,z]$ (both with initial point $y$), and $\mathrm{seg}[-y,z]$ and $\mathrm{seg}[-y,x]$ (starting at $-y$) to obtain $||y-z|| \geq ||y-x||$ and $||y+x|| \leq ||y+z||$. It follows that $||y+z|| - ||y-z|| \leq ||y+x|| - ||y-x||$. Moreover, equality holds if and only if $||y+z|| = ||y+x||$ and $||y - z|| = ||y - x||$. But this implies that $\mathrm{seg}\left[\frac{z-y}{||z-y||},x\right] \subseteq S$ and $\mathrm{seg}\left[\frac{x+y}{||x+y||},z\right]\subseteq S$, respectively. It is easy to see that this contradicts the convexity of the unit ball, and the proof is finished.

\end{proof}

\subsection{The q-angle and the S-angle}

Throughout this subsection, $(X,||\cdot||)$ is a normed plane. Only at the end we will extend the defined angles to spaces of higher dimensions.\\

The q-angle function $\mathrm{ang_q}:X_o\times X_o \rightarrow \mathbb{R}$ was defined in \cite{zhi2011projections} as follows: We first define a functional $q:X_o\times X_o \rightarrow \mathbb{R}$ by
\begin{align*}  q(x,y) = \left\{\begin{array}{ll} 0, \ \mathrm{if} \ x \ \mathrm{and} \ y \ \mathrm{are \ linearly \ dependent} \\ ||\mathrm{proj}_{x,y}||_{L(X)}^{-1}, \ \mathrm{if} \ x \ \mathrm{and} \ y \ \mathrm{are \ linearly \ independent}\end{array}\right., \end{align*}
where $\mathrm{proj}_{x,y}: X\rightarrow X$ is the projection over the line $\left<-x,x\right>$ taken in the direction of $y$; and $||\cdot||_{L(X)}$ is the standard norm in the space of linear transformations $T:X \rightarrow X$ given by
\begin{align*} ||T||_{L(X)} = \sup_{x\neq 0}\frac{||T(x)||}{||x||}. \end{align*}
It is clear that $0 \leq q(x,y) \leq 1$, with equality on the left side if and only if $x$ and $y$ are dependent, and equality on the right side if and only if $x \dashv_B y$. We now define the q-angle to be
\begin{align*} \mathrm{ang_q}(x,y) = \mathrm{arcsin}(q(x,y)), \end{align*}
where $\mathrm{arcsin} = \sin^{-1}:[0,1] \rightarrow \left[0,\frac{\pi}{2}\right]$. The q-angle has properties of continuity, parallelism and homogeneity type, and it preserves Birkhoff orthogonality (cf. \cite{zhi2011projections}, Theorems 2.1 and 2.2). The problem with this angle concept is that it does not make distinction between ``acute" and ``obtuse" angles, and for this reason we propose a different angle function based on the q-angle which we call the S-angle. First we notice that the functional $q$ is precisely the sine function $s:X_o\times X_o \rightarrow \mathbb{R}$ defined by
\begin{align*} s(x,y) = \frac{\inf_{t\in\mathbb{R}}||x+ty||}{||x||}, \end{align*}
which was studied in \cite{szostok} and \cite{bmt}. After having shown this, we can use all the ``machinery" from these references.

\begin{lemma}\label{qsine} In any normed plane the functional $q$ is precisely the sine function defined above.
\end{lemma}
\begin{proof} It is clear that both functionals are homogeneous. Thus, proving that $q(x,y) = s(x,y)$ whenever $x,y \in S$ is enough. It is also  immediate that the two functionals coincide for dependent vectors; hence we  assume that $x,y$ are linearly independent. The norm $||\mathrm{proj}_{x,y}||_{L(X)}$ is attained for any $z \in S$ such that $z \dashv_B y$. Indeed, if such a point is written as $z = \alpha x + \beta y$, then the unit circle lies within the (closed) strip determined by the lines $t \mapsto \alpha x + ty$ and $t \mapsto -\alpha x + ty$, and it follows that $||\mathrm{proj}_{x,y}(z)|| = |\alpha|$. On the other hand, using the formula $s(x,y) = \frac{|[x,y]|}{||y||_a||x||}$ (proved in \cite{bmt}), we have
\begin{align*} s(x,y) = \frac{|[x,y]|}{||y||_a} = \frac{|[x,y]|}{|[z,y]|} = \frac{1}{|\alpha|},\end{align*}
and the proof is finished.

\end{proof}

From now on we assume that the normed plane $(X,||\cdot||)$ is strictly convex. Having this hypothesis, we can have uniqueness for Birkhoff orthogonality on the left, i.e., for each given direction $y$ there exists precisely one direction $x$ such that $x \dashv_B y$. In this case we can define a map $b:X_o \rightarrow S$ which associates each $y \in X_o$ to the unique unit vector $b(y)$ satisfying $b(y) \dashv_B y$ and $[y,b(y)] > 0$. Now we define the S-angle $\mathrm{ang_s}:X_o\times X_o \rightarrow \mathbb{R}$ to be
\begin{align*}
\mathrm{ang_s}(x,y) = \left\{\begin{array}{ll} \mathrm{arcsin}(s(x,y)), \ \mathrm{if} \ [x,b(y)] \geq 0 \\ \pi - \mathrm{arcsin}(s(x,y)) \ \mathrm{otherwise} \end{array}\right..\end{align*}

It is easy to see that the S-angle satisfies the structural axioms 3 (homogeneity) and 5 (non-degeneracy) as well as the positional properties 6 (parallelism) and 8 (opposite invariance). Also, it is clear that the S-angle preserves Birkhoff orthogonality. Notice that continuity is not straightforward, and this we prove now.

\begin{lemma} For the $S$-angle the structural axiom 1 (continuity) holds.
\end{lemma}
\begin{proof} It is clear that $\mathrm{ang_s}$ is continuous for any pair $x,y \in X_o$ such that $[x,b(y)] \neq 0$ (this follows from the continuity of the sine function proved in \cite{bmt}). Hence, we just have to consider a pair $x,y$ with $x \dashv_B y$ and prove that $\mathrm{ang_s}$ is continuous in both entries. But for this sake we just have to restrict $\mathrm{ang_s}$ to $S\times S$ (homogeneity guarantees that this can be done) and consider the lateral limits over the unit circle.

\end{proof}

We come now to characterizations given by the S-angle. More precisely, we will characterize Radon and Euclidean planes by properties of the S-angle.

\begin{prop} The S-angle of a strictly convex normed plane $(X,||\cdot||)$ is symmetric if and only if the norm is Radon.
\end{prop}
\begin{proof} It is known (see \cite{bmt}) that the sine function is  symmetric only in Radon planes. The result follows.

\end{proof}

\begin{prop} The S-angle of a normed plane satisfies the structural axiom 4 (additivity) if and only if the norm is Euclidean.
\end{prop}
\begin{proof} First we shall check that if the S-angle is additive, then it is symmetric (and consequently the plane is Radon). Indeed, if additivity holds, then we have for $x,y \in S$
with $x \dashv_B y$ that
\begin{align*} \pi = \mathrm{ang_s}(x,-x) = \mathrm{ang_s}(x,y) + \mathrm{ang_s}(y,-x) = \frac{\pi}{2} + \mathrm{ang_s}(y,-x), \end{align*}
from which we get $\mathrm{ang_s}(y,-x) = \frac{\pi}{2}$ or, equivalently, $y \dashv_B x$.\\

Still assuming that $x,y \in S$ are two unit vectors such that $x \dashv_B y$, recall that the Busemann angular bisector of the angle $\wk\mathbf{xoy}$ is the ray with direction $x+y$, and hence Proposition 3.8 in \cite{bmt} states that $s(x,x+y) = s(y,x+y)$. If $\mathrm{ang_s}$ is additive, we have
\begin{align*} \mathrm{ang_s}(x,x+y) + \mathrm{ang_s}(x+y,y) = \mathrm{ang_s}(x,y) = \frac{\pi}{2}. \end{align*}
But since we have symmetry, it follows that $\mathrm{ang_s}(x,x+y) = \frac{\pi}{4}$. Doing the same for the angle $\wk\mathbf{xo(-y)}$ we have also $\mathrm{ang_s}(x,x-y) = \frac{\pi}{4}$. It follows that $||x+y|| = ||x-y||$. Thus, Birkhoff orthogonality implies isosceles orthogonality. This characterizes inner product spaces.

\end{proof}

Of course the q-angle and the S-angle definitions can be extended to normed spaces of higher dimensions, suitably considering the geometries induced in 2-dimensional subspaces spanned by pairs of independent vectors.

\subsection{The B-angle} \label{bangle}

In \cite{milicic2007b} Mili\v{c}i\v{c} defined, for smooth and strictly convex normed spaces, the so-called \textit{oriented B-angle}, which is an angle function preserving Birkhoff orthogonality. We basically follow his construction, but modify it by dropping the smoothness hypothesis and defining the oriented B-angle for any strictly convex normed space (still keeping its ``good" properties). Hence, within this subsection the normed space $(X,||\cdot||)$ is always assumed to be strictly convex.\\

Given $x,y \in X_o$, we assume $t^* = t^*(x,y)$ to be the (unique, by strict convexity) number for which $||x-t^*y|| = \inf_{t\in\mathbb{R}}||x-ty||$, and we study some properties of this functional.

\begin{lemma}[Properties of $t^*$]\label{propertiest*} Let $t^*:X_o\times X_o \rightarrow \mathbb{R}$ be the functional defined as above. We have the following properties: \\
\normalfont

\noindent\textbf{(a)} \textit{We have $t^*(x,\alpha x) = \frac{1}{\alpha}$ for every $\alpha \neq 0$. In particular, $t^*(x,x) = 1$ and $t^*(x,-x) = -1$. Moreover, $t^*(\alpha x,\beta y) = \frac{\alpha}{\beta}t^*(x,y)$ for any $\alpha, \beta \neq 0$}. \\

\noindent\textbf{(b)} \textit{We have $t^*(x,y) = 0$ if and only if $x \dashv_B y$.} \\

\noindent\textbf{(c)} \textit{The functional $t^*$ is continuous.}
\end{lemma}
\begin{proof} Assertions \textbf{(a)} and \textbf{(b)} are immediate, so we just have to prove that $t^*$ is continuous. Let $(x,y) \in X_o \times X_o$, and assume that $(x_n,y_n)_{n\in\mathbb{N}}$ is a sequence in $X_o \times X_o$ which converges to $(x,y)$. First, notice that $t^*(x_n,y_n)$ is given implicitly by
\begin{align*} ||x_n - t^*(x_n,y_n)y_n|| = s(x_n,y_n)||x_n||\cdot||y_n||_a, \end{align*}
where $s:X_o \times X_o \rightarrow \mathbb{R}$ is the (continuous) sine function studied in \cite{bmt}. Hence, since the right side of the equality above is continuous in both entries and from the uniqueness of $t^*$, it follows that any converging subsequence of $(t^*(x_n,y_n))_{n\in\mathbb{N}}$ must have $t^*(x,y)$ as limit. Therefore, in order to prove that $t^*$ is continuous, we only have to show that the sequence $(t^*(x_n,y_n))_{n\in\mathbb{N}}$ has some converging subsequence. This is an easy consequence of the inequality $|t^*(x,y)| \leq \frac{2||x||}{||y||}$, which comes directly from the triangle inequality.

\end{proof}

Given $x,y \in X_o$ such that $x$ is not Birkhoff orthogonal to $y$, strict convexity guarantees that there exists a unique non-zero number $t^{**} = t^{**}(x,y)$ for which $||x-t^{**}y|| = ||x||$. Geometrically this means that if $x$ is not Birkhoff orthogonal to $y$, then the line $t \mapsto x - ty$ intersects $S(||x||)$ in precisely two points. One of them is $x$, and by our definition the other one is $x - t^{**}y$. Finally we extend the definition to non-zero Birkhoff orthogonal vectors setting $t^{**}(x,y) = 0$ whenever $x \dashv_B y$.

\begin{lemma}[Properties of $t^{**}$]\label{propertiest**} The functional $t^{**}:X_o\times X_o \rightarrow \mathbb{R}$ has the following properties: \\

\normalfont\noindent\textbf{(a)} \textit{We have $t^{**}(x,\alpha x) = \frac{2}{\alpha}$ whenever $\alpha \neq 0$. Moreover, $t^{**}(\alpha x,\beta y) = \frac{\alpha}{\beta}t^{**}(x,y)$ for every $\alpha, \beta \neq 0$.}\\

\noindent\textbf{(b)}  \textit{We have $t^{**}(x,y) = t^*(x,y)$ if and only if $x \dashv_B y$.}\\

\noindent\textbf{(c)}  \textit{The functional $t^{**}$ is continuous.}

\end{lemma}

\begin{proof} Statement \textbf{(a)} is immediate. For \textbf{(b)} we have to notice that if $||x|| = ||x-t^{**}y|| = ||x-t^*y|| = \inf_{t\in\mathbb{R}}||x-ty||$, then $t^{**} = 0$ or $||x|| = ||x-ty||$ for all $t \in [0,t^{**}]$. The second option would contradict the strict convexity hypothesis if $t^{**} > 0$. \\

We have to invest more into the proof of \textbf{(c)}, although the result seems to be ``intuitively clear''. We prove that $t^{**}$ is continuous in each of its entries. First, we assume that $y \in X_o$ is fixed and let $x \in X_o$ be such that $x$ is not Birkhoff orthogonal to $y$. Let $(x_n)_{n\in\mathbb{N}}$ be a sequence in $X_o$ which converges to $x$. It is easy to see that if $n$ is sufficiently large, we have that $x_n$ is not Birkhoff orthogonal to $y$. Hence we may assume that $t^{**}(x_n,y) \neq 0$ for every $n \in \mathbb{N}$. If the sequence $t^{**}(x_n,y)$ has some subsequence which converges to $t_0 \in \mathbb{R}$, say, then we must have $t_0 \neq 0$ (recall that $x$ is not Birkhoff orthogonal to $y$), and since $||x_n-t^{**}(x_n,y)y|| = ||x_n||$, it follows from the continuity of the norm that $||x - t_0y|| = ||x||$. From the uniqueness we have $t_0 = t^{**}(x,y)$. To show that $(t^{**}(x_n,y))_{n\in\mathbb{N}}$ has some converging subsequence, it is enough to notice that the inequality $|t^{**}(x,y)| \leq \frac{2||x||}{||y||}$ holds. Let now $x,y \in X_o$ be such that $x \dashv_B y$, and let $(x_n)_{n\in\mathbb{N}}$ be a sequence in $X_o$ which converges to $x$. Since $||x_n - t^{**}(x_n,y)y|| = ||x_n||$ for every $n \in \mathbb{N}$, it is clear that any converging subsequence of $t^{**}(x_n,y)$ has to converge to $0$. Otherwise we would have $||x - t_0y|| = ||x||$ for some $t_0 \neq 0$, and this contradicts the hypothesis $x \dashv_B y$. \\

We are going to prove now that $t^{**}$ is continuous in the second entry. First, if $x \dashv_B y$ and $(y_n)_{n\in\mathbb{N}}$ is a sequence converging to $y$, then we have to show that $t^{**}(x,y_n)$ goes to $0$ when $n \rightarrow \infty$. We already know that $t^{**}(x,y_n)$ has a converging subsequence, and then it suffices to show that any of such subsequences must converge to $0$. Indeed, assume that $t_0$ is the limit of a converging subsequence of $t^{**}(x,y_n)$. Since $||x - t^{**}(x,y_n)y_n|| = ||x||$ for every $n \in \mathbb{N}$, it follows that $||x - t_0y|| = ||x||$, and from $x \dashv_B y$ we have $t_0 = 0$. Now, if $x$ is not Birkhoff orthogonal to $y$ and $y_n \rightarrow y$, then it is easy to see that $t^{**}(x,y_n)$ is bounded away from zero for large $n \in \mathbb{N}$. It follows that any subsequence of it converges to some non-zero number $t_0$ such that $||x - t_0y|| = ||x||$. Hence, by the property of strict convexity it follows  that $t_0 = t^{**}(x,y)$.

\end{proof}

These functionals yield the following characterization of inner product spaces.

\begin{prop}\label{ipschart*t**} A strictly convex normed space $(X,||\cdot||)$ is Euclidean if and only if $t^*(x,y) = \frac{t^{**}(x,y)}{2}$ for every $x,y \in X_o$.
\end{prop}
\begin{proof} If $(X,||\cdot||)$ is not Euclidean, then there exist unit vectors $x,y \in S$ such that $x$ is isosceles orthogonal, but not Birkhoff orthogonal to $y$. From isosceles orthogonality we have $t^{**}(x+y,y) = 2$. If $t^*(x+y,y) = \frac{t^{**}(x,y)}{2} = 1$, then $||x|| = ||x + y - y|| = \inf_{t\in\mathbb{R}}||x-ty||$, and thus $x \dashv_B y$. This contradicts the hypothesis, and the converse is obvious.

\end{proof}

Using the functionals $t^*$ and $t^{**}$, we define a new functional $\lambda:X_o\times X_o \rightarrow \mathbb{R}$ by \\
\begin{align*} \lambda(x,y) = \min\{|t^*(x,y)|,|t^{**}(x,y)-t^*(x,y)|\},
\end{align*}
and we explore some properties of this functional. Notice that, geometrically, $\lambda$ is the scaled length of the smallest of the segments determined on the chord of $S(||x||)$ through $x$ in direction of $y$ by the ray whose intersection with $S(||x||)$ is a point the direction $y$ of which supports this sphere.

\begin{lemma}\label{propertieslambda} For the functional $\lambda:X_o\times X_o \rightarrow \mathbb{R}$ the following holds: \\

\normalfont\noindent\textbf{(a)}  \textit{The functional $\lambda$ is continuous.} \\

\noindent\textbf{(b)}  \textit{We have $\lambda(\alpha x,\beta y) = \frac{|\alpha|}{|\beta|}\lambda(x,y)$ for all $\alpha,\beta \neq 0$.}\\

\noindent\textbf{(c)} \textit{ We have $\lambda(x,y) \leq \frac{||x||}{||y||}$, and equality holds if and only if $y = \alpha x$ for some $\alpha \neq 0$. }\\

\noindent\textbf{(d)}  \textit{We have $\lambda(x,y) = 0$ if and only if $x \dashv_B y$.}
\end{lemma}
\begin{proof} The assertions \textbf{(a)} and \textbf{(b)} are easy consequences of Lemmas \ref{propertiest*} and \ref{propertiest**}. For \textbf{(c)}, suppose that $|t^*(x,y)| > \frac{||x||}{||y||}$. From the definition of $t^{**}$ and the triangle inequality we have
\begin{align*} |t^{**}(x,y)|\cdot||y|| \leq ||x|| + ||x - t^{**}(x,y)y|| = 2||x||,  \end{align*}
and hence $|t^{**}(x,y)| \leq \frac{2||x||}{||y||}$. Also, from convexity we always have $|t^*(x,y)| \leq |t^{**}(x,y)|$ and $\mathrm{sgn}(t^*(x,y)) = \mathrm{sgn}(t^{**}(x,y))$, where $\mathrm{sgn}$ is the usual sign function. Thus,
\begin{align*} |t^{**}(x,y) - t^*(x,y)| = |t^{**}(x,y)| - |t^*(x,y)| \leq \frac{2||x||}{||y||} - \frac{||x||}{||y||} = \frac{||x||}{||y||}.
\end{align*}
This concludes the proof of the inequality. Now, if $|t^*(x,y)| = \frac{||x||}{||y||}$, then we have $|t^{**}(x,y) - t^*(x,y)| \leq \frac{||x||}{||y||}$, since $|t^{**}(x,y)| \leq 2\frac{||x||}{||y||}$ and $\mathrm{sgn}(t^{**}) = \mathrm{sgn}(t^*)$. On the other hand, if $|t^{**}(x,y) - t^*(x,y)| = \frac{||x||}{||y||}$, then $|t^*(x,y)| \leq \frac{||x||}{||y||}$. Hence, if equality holds in the above inequality, we must have $|t^{**}(x,y)| = 2\frac{||x||}{||y||}$ or $|t^*(x,y)| = \frac{||x||}{||y||}$. In both cases, strict convexity guarantees that $x$ and $y$ are dependent.\\

For the remaining part notice that if $x \dashv_B y = 0$, then we have $t^*(x,y) =  0$. Conversely, if $\min\{|t^*(x,y)|,|t^{**}(x,y)-t^*(x,y)|\} = 0$, then $t^*(x,y) = 0$ or $t^{**}(x,y) = t^*(x,y)$. It is easy to see that both cases imply $x \dashv_B y$.

\end{proof}

Now we are ready do define the oriented B-angle function $\mathrm{ang_b}:X_o\times X_o \rightarrow \mathbb{R}$. It is given by
\begin{align*} \mathrm{ang_b}(x,y) = \mathrm{arccos}\left(\lambda(x,y)\frac{||y||}{||x||}\mathrm{sgn}(t^{**}(x,y))\right). \end{align*}
There are two differences between our definition and  Mili\v{c}i\v{c}'s original one: we changed the order of the entries (for a merely aesthetic reason) and, more important, we replaced $\mathrm{sgn}(g(x,y))$ by $\mathrm{sgn}(t^{**}(x,y))$, where $g$ is a functional that we will study later (notice that we could have used $\mathrm{sgn}(t^*(x,y))$ instead). This allowed us to define the B-angle for non-smooth spaces, but in such a way that our definition agrees with the original one for smooth spaces. (Indeed, in a smooth and strictly convex normed space we  have $\mathrm{sgn}(g(x,y)) = \mathrm{sgn}(t^{**}(x,y))$ for any $x,y \in X_o$.)\\

The correctness of the definition of the oriented B-angle follows from Lemma \ref{propertieslambda} \textbf{(c)}. We study now its properties.

\begin{lemma}\label{propertiesbangle} The B-angle defined above satisfies the structural axioms 1 (continuity), 3 (homogeneity), 5 (non-degeneracy) and the three positional properties (parallelism, supplementarity and opposite invariance).
\end{lemma}

\begin{proof} Homogeneity follows immediately from Lemmas \ref{propertiest*} and \ref{propertiest**}, and opposite invariance is an easy consequence of it. \\

For continuity we have to prove that $\mathrm{ang_b}$ is continuous in the pairs $(x,y) \in X_o\times X_o$ such that $x \dashv_B y$. Indeed, all involved functions are continuous in pairs $(x,y)$ such that $x$ is not Birkhoff orthogonal to $y$. It is clear that if $x \dashv_B y$, we have that $|t^*(x_n,y_n)|$ and $|t^{**}(x_n,y_n) - t^*(x_n,y_n)|$ converge to $0$ whenever $(x_n,y_n)$ converges to $(x,y)$, and then the desired holds. \\

Non-degeneracy and parallelism follow from Lemma \ref{propertieslambda} \textbf{(c)} and from the fact that $\mathrm{sgn}(t^{**}(x,\alpha x)) = \mathrm{sgn}(\alpha)$. Supplementarity follows immediately from Lemma \ref{propertieslambda} \textbf{(b)} and from $\mathrm{sgn}(t^{**}(x,y)) = -\mathrm{sgn}(t^{**}(-x,y))$.

\end{proof}

\begin{prop} If $(X,||\cdot||)$ is a strictly convex normed space with associated B-angle $\mathrm{ang_b}:X_o\times X_o \rightarrow \mathbb{R}$, then we have $x \dashv_B y$ if and only if $\mathrm{ang_b}(x,y) = \frac{\pi}{2}$.
\end{prop}

\begin{proof} This comes immediately from Lemma \ref{propertieslambda} and from the definition of $t^{**}$.

\end{proof}

\subsection{The g-angle}

 In this subsection we will study an angle function based on a functional which, in some sense, is a generalization of the inner product for normed spaces. This is the functional $g:X\times X \rightarrow \mathbb{R}$ defined by
\begin{align*} g(x,y) = \frac{||x||}{2}\left(\tau_{-}(x,y) + \tau_{+}(x,y)\right), \end{align*}
where
\begin{align*} \tau_{\pm} = \lim_{t\rightarrow 0^{\pm}}\frac{1}{t}\left(||x+ty||-||x||\right). \end{align*}
Notice that this is based on the notion of Gateaux derivative, and that we have also an associated orthogonality concept. We say that a vector $x$ is \textit{g-orthogonal} to a vector $y$, writing $x \dashv_g y$, whenever $g(x,y) = 0$ (other concepts of orthogonality based on the functional $g$ will be introduced later).\\

This functional was extensively studied by Mili\v{c}i\v{c} \cite{milicic1973produit, milicic1987gortogonalite, milicic1990fonctionelle, milicic1987generalisation, milicic1998generalization}. We study now some of its properties.

\begin{prop}[Properties of g]\label{propg} In any normed space $(X,||\cdot||)$ the following properties hold for every $x,y \in X$ and $\alpha,\beta \in \mathbb{R}$. \\

\normalfont\noindent\textbf{(a)} \textit{$g(x,x) = ||x||^2$,} \\

\noindent\textbf{(b)} \textit{$g(\alpha x,\beta y) = \alpha\beta g(x,y)$,} \\

\noindent\textbf{(c)} \textit{$g(x,x+y) = ||x||^2 + g(x,y)$, and} \\

\noindent\textbf{(d)} \textit{$|g(x,y)| \leq ||x||\cdot||y||$.}
\end{prop}

The proof of this proposition is easy, and the reader may find it in the references listed above. As a consequence, we have immediately that $g$ is a \emph{semi-inner product in the sense of Lumer-Giles} (see \cite{dragomirsemiinner}). Moreover, it is clearly norm generating, i.e., $g(x,x) = ||x||^2$ for every $x \in X$. It is also worth noticing that in smooth normed spaces the functional $g$ can be defined in a simpler way, namely by

\begin{align*} g(x,y) = ||x||\lim_{t\rightarrow 0}\frac{||x+ty||-||x||}{t}. \end{align*}

Indeed, in such a space the norm is Gateaux differentiable (see \cite{james1947orthogonality}). In this case, $g$ is the unique norm generating a semi-inner product. Also, in these spaces the orthogonality $\dashv_g$ is equivalent to Birkhoff orthogonality (cf. \cite{milicic1998generalization} and \cite{milicic2011singer}).\\

We define the g-angle function $\mathrm{ang_g}:X_o\times X_o \rightarrow \mathbb{R}$ to be
\begin{align*} \mathrm{ang_g}(x,y) = \mathrm{arccos}\frac{g(x,y)}{||x||\cdot||y||}. \end{align*}
Proposition \ref{propg} guarantees that such a function is well defined. It is clear that this angle function preserves g-orthogonality and, consequently, also Birkhoff orthogonality if the space is smooth. \\

For the g-angle, only the structural axiom 3 (homogeneity) and the positional property 8 (opposite invariance) are immediate. Also it is clear that $\mathrm{ang_g}(\alpha x,\beta y) = \pi - \mathrm{ang_g}(x,y)$ if $\mathrm{sgn}(\alpha\beta) < 0$. Despite not necessarily having 5 (non-degeneracy) and 6 (parallelism), we can say that writing $x = \alpha y$, then $\mathrm{ang_g}(x,y) = 0$ if $\alpha > 0$ and $\mathrm{ang_g}(x,y) = \pi$ if $\alpha < 0$. We also can ensure continuity in the second variable, as the next lemma shows. Regarding continuity we have the following

\begin{lemma} The g-angle is continuous in the second variable for any normed space $(X,||\cdot||)$, and continuous in the first one if and only if the space is smooth.
\end{lemma}
\begin{proof} One can readily see that the following holds:
\begin{align*} \tau_{\pm}(x,y) - ||z|| \leq \tau_{\pm}(x,y+z) \leq \tau_{\pm}(x,y) + ||z||\,. \end{align*}
Hence, for a converging sequence $y_n \rightarrow y$ we may write
\begin{align*} g(x,y) - ||y_n-y|| \leq g(x,y + (y_n-y)) \leq g(x,y) + ||y_n-y||, \end{align*}
and this yields $g(x,y) \leq \lim_{n\rightarrow\infty}g(x,y_n) \leq g(x,y)$. This shows that the functional $g$ is continuous in the second entry, and therefore also is $\mathrm{ang_g}$. For the other statement we refer the reader to \cite[Theorem 13]{dragomirsemiinner}.

\end{proof}

 Using the functional $g$, Mili\v{c}i\v{c} defined a \textit{quasi-inner product space} to be a normed space where the equality
\begin{align*} ||x+y||^4 - ||x-y||^4 = 8\left(||x||^2g(x,y) + ||y||^2g(y,x)\right) \end{align*}
holds for every $x,y \in X$. This is a generalization of the parallelogram equality for inner product spaces, and of course any Euclidean space is a quasi-inner product space (see \cite{milicic1998generalization} for an example of non-trivial quasi-inner product space). In these spaces we have two other g-based orthogonality types which coincide with already known orthogonality concepts, and we introduce them now. Two vectors $x,y\in X$ are said to be \textit{symmetrically g-orthogonal} if $g(x,y) + g(y,x) = 0$. In this case we write $x \dashv_{g,s} y$. It is easy to see that in a quasi-inner product space the symmetric g-orthogonality is equivalent to Singer orthogonality. We also say that a vector $x \in X$ is \textit{isosceles g-orthogonal} to $y \in X$ if $||x||^2g(x,y) + ||y||^2g(y,x) = 0$; this is denoted by $x \dashv_{g,i} y$. One can easily check that in a quasi-inner product space isosceles orthogonality and isosceles g-orthogonality are equivalent, justifying the name that we have chosen. Preserving these g-orthogonality types correspondingly, we can define the angle functions $\mathrm{ang_{g,s}}:X_o\times X_o \rightarrow \mathbb{R}$ and $\mathrm{ang_{g,i}}:X_o\times X_o \rightarrow \mathbb{R}$ by
\begin{align*} \mathrm{ang_{g,s}}(x,y) = \mathrm{arccos}\frac{g(x,y)+g(y,x)}{2||x||\cdot||y||} \ \mathrm{and} \\ \mathrm{ang_{g,i}}(x,y) = \mathrm{arccos}\frac{||x||^2g(x,y) + ||y||^2g(y,x)}{||x||\cdot||y||\left(||x||^2+||y||^2\right)}\,, \end{align*}
respectively. Properties of these angles may be investigated by using similar methods as they were used for the usual g-angle. The angle function $\mathrm{ang_{g,s}}$ was studied in \cite{milicic1993gangle} and used in \cite{miliciccharacterizations} for characterizing convexity types for the unit sphere, and the angle function $\mathrm{ang_{g,i}}$ was defined and studied in \cite{milicic2007b}.\\

We mention also that using the uniqueness of a norm generating a semi-inner product in smooth normed spaces, Balestro and Shonoda \cite{Ba-Sho} concluded that the semi-inner product $g$ yields a known cosine-type function (namely, the function $\mathrm{cm}$ studied in \cite[Chapter 8]{Tho}), in a similar way as the inner product yields the usual cosine in Euclidean spaces. This can also be interpreted as a non-local definition for the Gateaux derivative of the norm, as it is clarified there.

\section{Angle measures} \label{anglemeasures}

\subsection{Brass' definition} \label{brassanglemeasures}

Brass \cite{brass} defined an angle measure in a normed plane to be a Borel measure $\mu$ on the unit circle $S$ satisfying\\

\noindent\textbf{(i)} $\mu(S) = 2\pi$,\\
\noindent\textbf{(ii)} for any Borel set $A \subseteq S$ it holds that $\mu(-A) = \mu(A)$, and\\
\noindent\textbf{(iii)} for each $v \in S$ we have $\mu(\{v\}) = 0$.\\

Following D\"{u}velmeyer \cite{Duev}, we add an additional non-degeneracy hypothesis: \\

\noindent\textbf{(iv)} Any non-degenerate arc of the unit circle has positive measure.\\

As first examples of such angular measures we have the measure given by length (in the norm) of corresponding arcs of the unit circle, as well as the measure determined by respective sector areas. Denoting them respectively as $\mu_l$ and $\mu_a$, we have the following

\begin{teo} In any Minkowski plane, $\mu_a$ is proportional to the measure in the unit circle given by the arc length in the antinorm.
\end{teo}

We refer the reader to \cite{martiniantinorms} for a proof. In particular, it follows that $\mu_a$ and $\mu_l$ are proportional in Radon planes. But these are not the only Minkowski planes with such property. D\"{u}velmeyer \cite{Duev} proved that the norms for which $\mu_a$ and $\mu_l$ are proportional are those whose unit circle is an \emph{equiframed curve}, which are defined to be centrally symmetric closed convex curves that are touched at each of their points by a circumscribed parallelogram of smallest area. A basic reference for equiframed curves is \cite{Ma-Sw2}. \\

Given a measure $\mu$ on $S$ as proposed by Brass, we define the associated angle function $\mathrm{ang_{\mu}}:X_o\times X_o\rightarrow \mathbb{R}$ to be the function which combines a pair of vectors $x,y \in X_o$ with the measure of the smaller arc determined on $S$ by the rays in the directions of $x$ and $y$. It is clear that any such angle function obeys the structural axioms 2 (symmetry), 3 (homogeneity), 4 (additivity), and 5 (non-degeneracy). We also have 1 (continuity), as a consequence of the following lemma, which comes from standard Measure Theory.

\begin{lemma}\label{lemma10s} Let $\lambda(S)$ be the length, in the norm, of the unit circle $S$ and consider the arc-length parametrization $p:\left[0,\frac{\lambda(S)}{2}\right] \rightarrow S$ of one of the arcs from $x_0$ to $-x_0$, where $x_0 \in S$ is any fixed vector. Then the \emph{mapping} $t \mapsto \mathrm{ang_{\mu}}(x_0,p(t))$ is continuous.
\end{lemma}

\noindent It is also easy to see that any angle function defined by an angle measure satisfies all positional properties. Recall that, given an angle function $\mathrm{ang}:X_o\times X_o\rightarrow\mathbb{R}$, we say that the measure of an angle $\wk\mathbf{xyz}$ formed by points $x,y,z \in X$ is $\mathrm{ang}(x-y,z-y)$. We continue with studying some geometric properties of a generic angle measure.

\begin{lemma}\label{lemma28} Fixing any angle measure $\mu$ in $S$, the sum of the interior angles of any triangle in $V$ equals $\pi$.
\end{lemma}

\begin{proof} Let $\Delta\mathbf{abc}$ be a triangle in $(V,||\cdot||)$. The union of the respectively smaller arcs from $\frac{b-a}{||b-a||}$ to $\frac{c-a}{||c-a||}$, from $\frac{c-a}{||c-a||}$ to $\frac{c-b}{||c-b||}$, and from $\frac{b-c}{||b-c||}$ to $\frac{b-a}{||b-a||}$ is a half-circle. Thus, additivity of the measure yields the desired.
\end{proof}

Regarding equilateral triangles, Brass proved the following theorem; see \cite{brass} or \cite{fankhaneldissertation} for a proof.

\begin{teo} Let $(X,||\cdot||)$ be a normed plane. Then there exists an angle measure $\mu$ for which every equilateral triangle is equiangular if and only if the norm is not rectilinear (i.e., if the unit circle is not a parallelogram).
\end{teo}

\subsection{I-measures}

Following \cite{fankhanel2009i}, this section is devoted to angle measures for which isosceles triangles have equal base angles. We call such a measure an I-measure, and we prove that such a measure is only possible in the Euclidean plane.

\begin{lemma}\label{lemma29} For any angle measure $\mu$ on $S$ the following properties are equivalent:\\

\normalfont\noindent\textbf{(i)} \textit{The measure of the angle $\wk\mathbf{vw(-v)}$ equals $\frac{\pi}{2}$ for any linearly independent $v,w \in S$.} \\

\noindent\textbf{(ii)} \textit{We have} \textit{$\mathrm{ang_{\mu}}(x,y) = \frac{\pi}{2}$ for any $x,y \in X_o$ with $x \dashv_I y$.}
\end{lemma}

\begin{proof} Assume that \textbf{(i)} holds and let $x,y$ be non-zero vectors which are isosceles orthogonal. Let $v = x-y$ and $w = x+y$. Since $||v|| = ||w||$, it follows that $v$ and $w$ are points of the same circle centered at the origin, and this yields that the measure of the angle $\wk\mathbf{vw(-v)}$ equals $\frac{\pi}{2}$. In other words, $\mathrm{ang_{\mu}}(v-w,-v-w) = \frac{\pi}{2}$. Now,
\begin{align*} \frac{\pi}{2} = \mathrm{ang_{\mu}}(v-w,-v-w) = \mathrm{ang_{\mu}}(-2y,-2x) = \mathrm{ang_{\mu}}(x,y). \end{align*}
For the converse, let \textbf{(ii)} hold and fix linearly independent vectors $v,w \in S$. Clearly, $(w+v) \dashv_I (w-v)$, and hence $\mathrm{ang_{\mu}}(w+v,w-v) = \frac{\pi}{2}$.  This finishes the proof, since $\mathrm{ang_{\mu}}(w+v,w-v)$ equals the measure of the angle $\wk\mathbf{vw(-v)}$.
\end{proof}

It is easy to see that any I-measure $\mu$ satisfies \textbf{(i)} (Thales' theorem) and consequently \textbf{(ii)}. Next we see how Birkhoff orthogonality behaves with respect to I-measures.

\begin{lemma} Let $x ,y \in X_o$ be Birkhoff orthogonal vectors, and assume that $\mu$ is an I-measure in $S$. Then $\mathrm{ang_{\mu}}(x,y) = \frac{\pi}{2}$.
\end{lemma}
\begin{proof} The proof is a continuity argument. Let $x \in S$ be an arbitrary point and assume that $x+y \in X\setminus\{x\}$ is a point of the supporting line to $S$ at $x$ (i.e., $x \dashv_B y$). Let $x_n \in S$ be a sequence of unit vectors such that $x_n \rightarrow x$. Hence we have that $\mathrm{ang_{\mu}}(x_n-x,-x) \rightarrow \mathrm{ang_{\mu}}(x,y)$. On the other hand, it is clear that $\mathrm{ang_{\mu}}(x,x_n) = 0$, and since $\mu$ is an I-measure we have $\mathrm{ang_{\mu}}(x+x_n,x-x_n) = \frac{\pi}{2}$. It follows that $\mathrm{ang_{\mu}}(-x,y) = \frac{\pi}{2}$. The proof is complete.

\end{proof}

\begin{teo} Let $(X,||\cdot||)$ be a normed plane endowed with an I-measure $\mu$. Then the norm is Euclidean, and $\mu$ is the standard Euclidean angle measure.
\end{teo}
\begin{proof} Assume that $x,y \in X_o$ are such that $x \dashv_I y$. Thus, we have $\mathrm{ang_{\mu}}(x,y) = \frac{\pi}{2}$. Let $z \in X_o$ be such that $||z|| = ||x||$ and $x \dashv_B z$. We have also $\mathrm{ang_{\mu}}(x,z) = \frac{\pi}{2}$. Since $\mu$ is non-degenerate, we have $y = z$ or $y = -z$. Hence isosceles orthogonality implies Birkhoff orthogonality. This characterizes the Euclidean plane. \\

We have to prove now that $\mu$ is the standard Euclidean angular measure. For this sake we just have to notice that $\mu$ preserves the standard measure for any $2^{-k}$-multiple of the straight angle, and to use the $\sigma$-additivity of the measure.

\end{proof}

\begin{remark} The non-degeneracy of $\mu$ can be dropped. In fact, the original proof of Fankh\"{a}nel assumes that $\mu$ is defined in the sense of Brass.
\end{remark}

\subsection{B-measures and T-measures}

The B-measures and T-measures are defined to be angle measures preserving Birkhoff and isosceles orthogonality, respectively. This means that $\mathrm{ang_{\mu}} = \frac{\pi}{2}$ whenever $x \dashv_B y$ and $x \dashv_I y$, respectively. About these types of angle measures Fankh\"{a}nel \cite{fankhaneldissertation} proved the following theorems.

\begin{teo} An arc $R$ of the unit circle of a normed plane is called a Radon arc whenever the following implication holds: if $x \in R$, $y \in S$ and $x \dashv_B y$, then $y \dashv_B x$. A normed plane whose unit circle contains a Radon arc admits a B-measure.
\end{teo}

\begin{teo} A normed plane has a T-measure if and only if it is Euclidean.
\end{teo}

\section{Further notions}

\subsection{Wilson angles}

Wilson \cite{wilson1932relation} extended angle concepts to arbitrary metric spaces, and Valentine and Andalafte (\cite{valentine1971wilson}) studied these concepts in real normed spaces, obtaining a characterization of inner product spaces by a local property. This section is devoted to give a simpler proof to the main result of \cite{valentine1971wilson}, which states that \emph{the Wilson angle can only be well defined in a normed space if its norm is derived from an inner product}.\\

Given a metric space $(M,d)$, Wilson's definition can be stated as a function $\wk_\mathrm{w}:\{x,y,z\in M: x,y,z \ \mathrm{are \ mutually \ distinct} \}\rightarrow \mathbb{R}$ given by
\begin{align*} \wk_{\mathrm{w}}(x,y,z) = \mathrm{arccos}\left(\frac{d(x,y)^2+d(y,z)^2-d(x,z)^2}{2d(x,y)d(y,z)} \right).\end{align*}

Using this ``3-point angle" definition, one may define the angle between two metric rays. Indeed, if $r,s\subseteq M$ are two metric rays with the same origin $y$, then we define the angle $\mathrm{ang_w}$ between them to be
\begin{align*} \mathrm{ang_w}(r,s) = \lim_{x,z\rightarrow y}\wk_{\mathrm{w}}(x,y,z), \end{align*}
where $x \rightarrow y$ through the ray $r$ and $z \rightarrow y$ through $s$. This limit may not exist, and we deal from now on with a normed space $(X,||\cdot||)$ which has the property that it exists for any pair of rays with common initial point. For simplicity, we will refer to this property as the \textit{Wilson property}. In this case, we can easily see that the angle between rays is translation invariant, and the rays can be replaced by non-zero vectors. Formally, we define the function $\mathrm{ang_w}:X_{o}\times X_{o} \rightarrow \mathbb{R}$ to be
\begin{align*} \mathrm{ang_w}(x,y) = \lim_{p,q\rightarrow o}\wk_{\mathrm{w}}(p,o,q) = \lim_{p,q\rightarrow o}\mathrm{arccos}\left(\frac{||p||^2+||q||^2-||p-q||^2}{2||p||\cdot||q||} \right),\end{align*}
where $p$ ranges through $\left.[o,x\right>$ and $q \in \left.[o,y\right>$. It is easy to see that the angle function $\mathrm{ang_w}$ defined this way is homogeneous.

\begin{lemma} In a normed space with the Wilson property we have $\mathrm{ang_w}(x,y) = \mathrm{ang_p}(x,y)$ for every $x,y \in X_{o}$.
\end{lemma}
\begin{proof} Clearly, we can write
\begin{align*} \mathrm{ang_w}(x,y) = \lim_{t\rightarrow 0}\mathrm{arccos}\left(\frac{||tx||^2+||ty||^2-||tx-ty||^2}{2||tx||\cdot||ty||}\right) = \\ = \mathrm{arccos}\left(\frac{||x||^2+||y||^2-||x-y||^2}{2||x||\cdot||y||}\right) = \mathrm{ang_p}(x,y),  \end{align*}
where we approximate considering that $t > 0$.

\end{proof}

\begin{coro}\label{wilsoninner} A normed space has the Wilson property if and only if it is an inner product space.
\end{coro}
\begin{proof} Since $\mathrm{ang_w}$ is homogeneous and $\mathrm{ang_w} = \mathrm{ang_p}$, we get the result from Proposition \ref{homogeneitypangle}.

\end{proof}

\subsection{The Diminnie-Andalafte-Freese angle} \label{d-a-f angle}

In \cite{D-A-F} the authors define an angle  function $\mathrm{ang}:X_o \times X_o \rightarrow \mathbb{R}$ by
\begin{align*} \mathrm{ang}(x,y) = \mathrm{arccos}\left(1-\frac{1}{2}\left|\left|\frac{x}{||x||}-\frac{y}{||y||}\right|\right|^2\right).
\end{align*}

For simplicity, we will call this the \textit{D-A-F angle}. It is based on the Euclidean cosine law for isosceles triangles whose equal sides are unit. Of course, any angle in Euclidean space can be obtained in this way. Regarding the structural axioms, it is clear that such an angle function satisfies 1 (continuity), 2 (symmetry), 3 (homogeneity), and 5 (non-degeneracy). The positional property 8 (opposite invariance) also holds. These axioms and properties will be used throughout this subsection with no further comments. Regarding the structure of this angle function, we also have (in the spirit of Theorem \ref{thureytheorem})

\begin{prop}\label{propdafangle} Let $x,y \in X_o$ be unit independent vectors. Then we have \\

\normalfont
\noindent\textbf{(a)} \textit{$\lim_{t\rightarrow-\infty}\mathrm{ang}(y,x+ty) = \pi$.}\\

\noindent\textbf{(b)} \textit{$\lim_{t\rightarrow+\infty}\mathrm{ang}(y,x+ty) = 0$.}\\

\noindent\textbf{(c)} \textit{For any $k \in (0,\pi)$ there exists an $\alpha \in\mathbb{R}$ such that $\mathrm{ang}(y,x+\alpha y) = k$.}\\

\textit{In other words, the map $t\mapsto \mathrm{ang}(y,x+ty)$ is a non-increasing surjective function from $\mathbb{R}$ onto $(0,\pi)$. Moreover, if $(X,||\cdot||)$ is strictly convex, then this map is a decreasing homeomorphism.}

\end{prop}

\begin{proof} It is easy to see that $\lim_{t\rightarrow+\infty}\frac{x+ty}{||x+ty||} = y$ and $\lim_{t\rightarrow-\infty}\frac{x+ty}{||x+ty||} = -y$, and this gives \textbf{(a)} and \textbf{(b)}. The assertion \textbf{(c)} comes immediately from the Intermediate Value Theorem. To prove that $t\mapsto\mathrm{ang}(y,x+ty)$ is non-increasing (or decreasing, in the strictly convex case) we notice that $t\mapsto\frac{x+ty}{||x+ty||}$ is a parametrization of the (open) half-circle from $-y$ to $y$, which contains $x$, and use the Monotonicity Lemma (see \cite{martini1}).

\end{proof}

As a first relation between the properties of the D-A-F angle and the structure of the normed space we have the following lemma.\\

\begin{lemma}\label{dafangleparallelism} The D-A-F angle of a normed space $(X,||\cdot||)$ has the positional property 6 (parallelism) if and only if the space is strictly convex.
\end{lemma}
\begin{proof} We have $\mathrm{ang}(x,y) = \pi$ if and only if
\begin{align*} \left|\left|\frac{x}{||x||}-\frac{y}{||y||}\right|\right| = 2,
\end{align*}
and then the desired follows immediately from the triangle inequality.

\end{proof}

Several characterizations of Euclidean spaces in terms of the D-A-F angle were proved by Diminnie, Andalafte and Freese (see \cite{D-A-F,F-D-A}) and by Martini and Wu \cite{Ma-Wu}. Some of them are based on properties of bisectors, and thus we will discuss them in Subsection \ref{bisectors}. The other ones rely on weaker versions of structural axioms and positional properties. We define these weaker versions now, but the reader may notice that we are changing their original names given in \cite[Definition 1.2]{D-A-F}. After that we will prove the mentioned characterizations on the lines of the proof presented in \cite{D-A-F}. \\

\begin{defi}\label{weakerpropoerties} Let $\sim$ denote one of the symbols $\leq$, $=$, or $\geq$.\\

\noindent\textbf{(i)} An angle function has the \textit{weak supplementarity} property if $\mathrm{ang}(x,y) + \mathrm{ang}(-x,y) \sim \pi$ for any $x,y \in X_o$.

\noindent\textbf{(ii)} We say that an angle is \textit{weakly additive} whenever $\mathrm{ang}(x,\alpha x+\beta y) + \mathrm{ang}(\alpha x+\beta y) \sim \mathrm{ang}(x,y)$ for all independent $x,y \in X_o$ and $\alpha,\beta > 0$.

\noindent\textbf{(iii)} An angle $\mathrm{ang}$ is said to have the \textit{angle sum property} if and only if $\mathrm{ang}(x,y)+\mathrm{ang}(y,y-x) + \mathrm{ang}(x,x-y) \sim \pi$ for all independent $x,y \in X_o$.

\noindent\textbf{(iv)} An angle function has the \textit{exterior angle property} when $\mathrm{ang}(y,y-x) + \mathrm{ang}(x,x-y) \sim \mathrm{ang}(-x,y)$ for any independent $x,y \in X_o$.\\
\end{defi}

\begin{teo}\label{teodafangle} Let $(X,||\cdot||)$ be a normed space, and let $\mathrm{ang}:X_o\times X_o\rightarrow\mathbb{R}$ denote its D-A-F angle function. The following are equivalent:\\

\noindent\textbf{(a)} $\mathrm{ang}(\cdot,\cdot)$ has the weak supplementarity property,\\

\noindent\textbf{(b)} $\mathrm{ang}(\cdot,\cdot)$ is weakly additive,\\

\noindent\textbf{(c)} $\mathrm{ang}(\cdot,\cdot)$ has the angle sum property,\\

\noindent\textbf{(d)} $\mathrm{ang}(\cdot,\cdot)$ has the exterior angle property, and\\

\noindent\textbf{(e)} $(X,||\cdot||)$ is an inner product space.\\
\end{teo}

\begin{proof}
It is clear that \textbf{(e)} implies all the other assertions. Hence we only have to prove the converse for each of these implications. \\

We start with $\mathbf{(a)}\Rightarrow\mathbf{(e)}$. Let $x,y \in X_o$ be unit vectors for which $\mathrm{ang}(x,y) + \mathrm{ang}(-x,y) \leq \pi$. Thus, $\cos(\mathrm{ang}(-x,y)) \geq \cos(\pi-\mathrm{ang}(x,y)) = -\cos(\mathrm{ang}(x,y)$. On the other hand, we have the equalities
\begin{align*} ||x-y||^2 = 2 - 2\cos(\mathrm{ang}(x,y)) \ \mathrm{and} \\ ||x+y||^2 = 2-2\cos(\mathrm{ang}(-x,y)).
\end{align*}
It follows that $||x-y||^2+||x+y||^2 \leq 4$. Using similar methods for the other cases it follows that \textbf{(a)} implies $||x-y||^2+||x+y||^2 \sim 4$ for any unit vectors $x,y \in X_o$. This is a characterization of inner product spaces (see \cite{Day}).\\

Now we prove that \textbf{(b)} implies \textbf{(a)} (and hence \textbf{(e)}). If weak additivity holds, then for any independent $x,y \in X_o$ and $\alpha,\beta > 0$ we have $\mathrm{ang}(x,y)+\mathrm{ang}(y,\beta y-\alpha x) \sim \mathrm{ang}(x,\beta y - \alpha x)$. By continuity and homogeneity, considering $\beta \rightarrow 0$ yields $\mathrm{ang}(y,\beta y-\alpha x) \rightarrow \mathrm{ang}(y,-x)$ and $\mathrm{ang}(x,\beta y - \alpha x) \rightarrow \pi$. Hence, $\mathrm{ang}(x,y)+\mathrm{ang}(-x,y) \sim \pi$. \\

Assume now that \textbf{(c)} holds. Thus, for any $t > 0$ we have $\mathrm{ang}(x,y)+\mathrm{ang}(y,y-tx)+\mathrm{ang}(x,tx-y) \sim \pi$. Since $\lim_{t\rightarrow+\infty}\frac{y-tx}{||y-tx||} = -\frac{x}{||x||}$ (see the proof of Proposition \ref{propdafangle}), we have
\begin{align*} \lim_{t\rightarrow+\infty}\mathrm{ang}(y,y-tx) = \lim_{t\rightarrow+\infty}\cos^{-1}\left(1-\frac{1}{2}\left|\left|\frac{y}{||y||}-\frac{y-tx}{||y-tx||}\right|\right|^2\right) = \\ = \mathrm{ang}(-x,y).
\end{align*}
From Proposition \ref{propdafangle}\textbf{(b)} we have $\lim_{t\rightarrow+\infty}\mathrm{ang}(x,tx-y) = \mathrm{ang}(x,x) = 0$, and thus it follows that $\mathrm{ang}(x,y) + \mathrm{ang}(-x,y) \sim \pi$. Therefore \textbf{(a)} holds, and so also \textbf{(e)}.\\

To complete the proof we show that if \textbf{(d)} holds, then $\mathrm{ang}(\cdot,\cdot)$ is weakly additive. Assuming that $\mathrm{ang}(\cdot,\cdot)$ has the exterior angle property, we have for any independent $x,y \in X_o$ and $\alpha, \beta > 0$
\begin{align*} \mathrm{ang}(x,\alpha x+\beta y) + \mathrm{ang}(\alpha x + \beta y,y) = \\= \mathrm{ang}(\alpha x,\alpha x+\beta y) + \mathrm{ang}(-\beta y,-\beta y-\alpha x)\sim \\\sim \mathrm{ang}(\alpha x,\beta y) = \mathrm{ang}(x,y),
\end{align*}
where we used homogeneity and opposite invariance. This confirms weak additivity.

\end{proof}

The D-A-F angle inspires an orthogonality relation which was also studied in \cite{D-A-F}. We define and study it now.

\begin{defi} Given two vectors $x,y \in X$, we say that $x$ is \textit{orthogonal} to $y$ (denoted simply by $x \dashv y$) if $||x|| \cdot ||y|| = 0$ or
\begin{align*} \left|\left|\frac{x}{||x||}-\frac{y}{||y||}\right|\right| = \sqrt{2}.
\end{align*}
\end{defi}

This orthogonality type is clearly symmetric and (positively) homogeneous. Also, as a consequence of Proposition \ref{propdafangle}\textbf{(c)} we have that for any $x,y \in X$ there exists some $\alpha \in \mathbb{R}$ such that $x \dashv (\alpha x + y)$.\\

In \cite{D-A-F} the authors point out that homogeneity of this orthogonality (in the sense that $x \dashv y$ implies $\alpha x \dashv \beta y$ for any $\alpha,\beta \in \mathbb{R}$) does not characterize inner product spaces, even if the considered space is strictly convex. As a counterexample the authors present a Minkowski plane whose unit circle is constructed by slightly deforming the sides of a regular octagon (notice that in the normed plane with unit circle given by a regular octagon the orthogonality relation is clearly homogeneous).

\subsection{Angular bisectors} \label{bisectors}

It is very natural that the concept of angular bisectors (= bisectors of angles, meant analogously to that notion as it occurs in Euclidean geometry) was also investigated for general normed spaces. Clearly, different concepts of angular bisectors are, of course, only acceptable if they coincide with Euclidean angular bisectors in the inner-product case.\\

Glogovskii \cite{Glo} defined the angular bisector of an angle $\wk\mathbf{aob}$ (thus called \emph{Glogovskii bisector}) in a Minkowski plane as the ray of all points each having the same distances to the rays $\left.[o,a\right>$ and $\left.[o,b\right>$, respectively. Thus, the three Glogovskii bisectors of a triangle have a point in common which is clearly the incenter of that triangle (see also \cite[\S 8.1]{martiniantinorms} and, for the Mannhattan norm, particularly \cite{Sow}). On the other hand, the angular bisector of $\wk\mathbf{aob}$ was defined by Busemann (see \cite{Bus5}) as follows: if $o$ is the origin and $a, b$ are boundary points of the unit ball of a normed plane, then the midpoint $d$ of $\mathrm{seg}[a,b]$ yields with $\left.[o,d\right>$ the \emph{Busemann bisector} of $\wk\mathbf{aob}$. In \cite{Bus5}, Busemann gives characterizations of normed planes within a more general class of planes via properties of angular bisectors defined in this way; see also Busemann \cite[Chapter 1]{Bus4} and Phadke \cite{Ph}. We note as well that (like the notion of Busemann bisector) the angular bisector property was also investigated in more general spaces, still yielding characterizations of (certain) normed spaces and, in particular, of inner product spaces; see \cite{A-F}.\\

The concurrence of Glogovskii bisectors of triangles still holds if one extends the framework of norms to general convex distance functions (or gauges), obtained from norms by deleting simply the symmetry axiom; see Guggenheimer \cite{Gug}. This concurrence of Glogovskii bisectors was also extended to simplices in higher dimensional normed spaces by Averkov \cite{Av}. D\"{u}velmeyer \cite{Due} showed that a normed plane is a Radon plane if and only if Glogovskii's and Busemann's definitions of angular bisectors coincide. A different proof of this fact can be found in \cite{bmt}. This result yields also a characterization of Radon curves that only involves basic vector space concepts; cf. \cite[Proposition 5.2]{Ba-Ma-Te}. D\"{u}velmeyer \cite{Duev} also studied the relations of Glogovskii and Busemann angular bisectors with bisectors given by angle measures in the sense of Brass (see Subsection \ref{brassanglemeasures}), the latter defined in the intuitive way. He proved the following

\begin{teo} Let $(X,||\cdot||)$ be a Minkowski plane endowed with an angle measure $\mu$ (in the sense of Brass, and with the non-degeneracy hypothesis presented in Subsection \ref{brassanglemeasures}). Then the following properties are equivalent:\\

\normalfont
\noindent\textbf{(a)} \textit{For any angle, the Glogovskii angular bisector coincides with the bisector given by $\mu$.}\\

\noindent\textbf{(b)} \textit{The Busemann angular bisector of any angle coincides with its $\mu$-angular bisector.}\\

\noindent\textbf{(c)} \textit{The plane is Euclidean and $\mu$ is the standard Euclidean angle measure.}
\end{teo}

In a normed space endowed with a symmetric and homogeneous angle function $\mathrm{ang}:X_o\times X_o\rightarrow \mathbb{R}$ we have a well defined system of bisectors if for any $x,y\in X_o$ there exists a unique (up to scaling) positive linear combination $z = \alpha x + \beta y$ for which  $\mathrm{ang}(x,z) = \mathrm{ang}(y,z)$. In \cite{F-D-A} bisectors of the D-A-F angle (see Subsection \ref{d-a-f angle}) were used for characterizing inner product spaces as follows.

\begin{teo} Let $(X,||\cdot||)$ be a normed space and let $\mathrm{ang}:X_o\times X_o\rightarrow\mathbb{R}$ be its D-A-F angle function. For any vectors $x,y \in X_o$ there exists, up to scaling, a unique positive linear combination $z = \alpha x+\beta y$ for which $\mathrm{ang}(x,z) = \mathrm{ang}(y,z)$. If for any unit vectors $x,y \in X_o$ we have that the positive combination $z = x+y$ is such that $\mathrm{ang}(x,z) = \mathrm{ang}(y,z) = \frac{1}{2}\mathrm{ang}(x,y)$, then the space is Euclidean.
\end{teo}

In the same paper the authors asked whether or not a certain weaker property for the bisectors of the D-A-F angle would characterize Euclidean spaces. The question was positively answered by Martini and Wu \cite{Ma-Wu}. They proved the following

\begin{teo} Let $(X,||\cdot||)$ be a Minkowski space with D-A-F angle $\mathrm{ang}(\cdot,\cdot)$. If for any independent vectors $x,y \in X_o$ the positive linear combination $z = \alpha x + \beta y$, for which $\mathrm{ang}(x,z) = \mathrm{ang}(y,z)$ holds, also accomplishes $\mathrm{ang}(x,z) = \frac{1}{2}\mathrm{ang}(x,y)$, then $||\cdot||$ is derived from an inner product.
\end{teo}

We finish this part by mentioning a natural concept related to angular bisectors, which is still waiting for its (promising) extension from the Euclidean norm to general normed planes (see \cite{Ai-Au} and \cite{A-A-A-G}): For a given polygon $P$, its \emph{straight skeleton} is created as the interference pattern of certain wavefronts propagated from the edges of $P$. We mention this here since strong relations of this concept to angular bisectors for non-Euclidean types of distance functions are confirmed; cf. \cite{B-D-G}.\\

\subsection{Other related concepts}

All the angle functions and angular measures discussed so far in this survey do not exhaust all the existing theory. We have chosen to explore in more detail the angle concepts that seem to be important and promising with respect to possible further research. In this subsection we briefly outline some further concepts that are, from our point of view, also important and therefore worth mentioning. \\

 On further angle concept that appears in the literature is that of the \emph{Dekster angle}, firstly defined in the paper \cite{Dek1}. It is not an angle function in the sense of Section \ref{abstractangles} nor, unless the dimension of the space is 2, an angle measure in the sense of Section \ref{anglemeasures}. To construct this angle, we first assume that $E^d$ is the usual $d$-dimensional Euclidean space, with unit sphere $S^{d-1}$, and that the Minkowski norm is given by a centrally symmetric $d$-dimensional convex body $B$ in the usual way. We set
\begin{align*}||v|| = \inf\{\lambda \in \mathbb{R}^+:x\in\lambda B\}, \end{align*}
for $x \in E^d$. We let $\pi:S^{d-1}\rightarrow\partial B$ be the projection from the origin, and we define $r:S^{d-1}\rightarrow\mathbb{R}$ to be $r(\varphi) = ||\pi(\varphi)||_{\mathrm{E}}$, where $||\cdot||_{\mathrm{E}}$ denotes the Euclidean length. Also, for every direction $\varphi \in S^{d-1}$ we denote by $E^{d-1}(\varphi)$ its orthogonal complement (in $E^d$) and by $b(\varphi)$ the $(d-1)$-dimensional Euclidean volume of the projection of $B$ onto $E^{d-1}(\varphi)$. Finally, if $U \subseteq \partial B$ is a closed domain, we define the \emph{angular measure} $\Phi(U)$ of the \emph{central cone} determined by $U$ to be
\begin{align*}\Phi(U) = \epsilon_{d-1}\int_{\Omega}\frac{r^{d-1}(\varphi)}{b(\varphi)}dS, \end{align*}
where $\epsilon_{d-1}$ is the volume of a unit ball in $E^{d-1}$, and $\Omega = \pi^{-1}(U)$. Dekster \cite[Theorem 1.2]{Dek1} proved that if $\Omega=\pi^{-1}(U)$ is a closed domain whose boundary is piecewise regular, then the value of $\Phi(U)$ is invariant under linear isomorphisms (considering the Minkowski space determined by the image of $B$ through such a map). Curvature concepts based on this angle definition were also investigated. In addition, the author defines the \emph{total angular measure around a point} to be
\begin{align*} \tau = \varphi(\partial B) = \epsilon_{d-1}\int_{S^{d-1}}\frac{r^{d-1}(\varphi)}{b(\varphi)}dS\,.\end{align*}
 Performing some calculations he proved that, for example, the total angular measure around a point is less than $2\pi$ for the Minkowski planes with unit circle given by a square and a regular hexagon, respectively. Ling \cite{Ling} proved that in two-dimensional spaces we always have the estimate $\sqrt{2\pi} \leq \tau \leq 8$. Moreover, in \cite{Dek2} Dekster's angle is used to endow the space of directions of a Minkowski space (which is topologically a sphere) with a certain metric. \\

\emph{Trigonometry} is, of course, closely related to angles. The first trigonometric function defined for Minkowski geometry seems to have been given by Finsler \cite{Fi2}. Later, Barthel \cite{Bar1}, Busemann \cite{Bus3} and Petty \cite{Pet} studied sine and cosine functions and used them to investigate concepts such as curvatures. Such trigonometric functions were also used by Guggenheimer \cite{Gug1,Gug2}, and the cosine one was recently revisited in \cite{Ba-Sho}. A good reference to them is \cite[Chapter 8]{Tho}. Trigonometric functions also appeared when Petty and Barry \cite{P-B} and Guggenheimer \cite{Gug3} investigated plane Minkowski geometry considering unit circles given by solutions of second order differential equations of type $f''(t) +p(t)f(t) = 0$, known as ``Hill equations" (see also \cite{Wal}). Recently, a different sine function was studied and used for defining geometric constants which can estimate, for example, the difference between orthogonality types (see \cite{bmt,Ba-Ma-Te,Ba-Ma-Te2}). \\

A very recent new concept related to angles is given in \cite{Ba-Ho-Ma}. The authors use a generalization of Brass' concept of angular measure to obtain groups of rotations for Minkowski planes. Using such rotations and translations one can define motion groups for Minkowski planes, with the disadvantage that such motions are not necessarily isometries. With these new tools, roulettes are defined and analogues to the Euler-Savary equations are obtained. \\

\bibliography{bibliography}

\end{document}